\documentclass[12pt]{article}

\usepackage{amsmath,amsthm,amsfonts,amssymb, color,xcolor,subcaption,graphicx} 

\newtheorem{theorem}{Theorem}[section]
\newtheorem{corollary}[theorem]{Corollary}

\newtheorem{assumption}[theorem]{Assumption}

\newtheorem{proposition}[theorem]{Proposition}
\newtheorem{lemma}[theorem]{Lemma}
\newtheorem{definition}[theorem]{Definition}

\def\cA{\mathcal{A}}
\def\cB{\mathcal{B}}
\def\cC{\mathcal{C}}

\def\cE{\mathcal{E}}
\def\cF{\mathcal{F}}
\def\cG{\mathcal{G}}
\def\cH{\mathcal{H}}
\def\cI{\mathcal{I}}

\def\cN{\mathcal{N}}
\def\cP{\mathcal{P}}
\def\cR{\mathcal{R}}

\def\cS{\mathcal{S}}

\def\cU{\mathcal{U}}

\def\bD{\mathbb{D}}
\def\bE{\mathbb{E}}
\def\bL{\mathbb{L}}

\def\bN{\mathbb{N}}
\def\bP{\mathbb{P}}
\def\bR{\mathbb{R}}

\def\e{\varepsilon}

\newcommand{\les}{\lesssim}

\newcommand{\s}{\sigma}

\begin{document}

\title{Gaussian fluctuations for the nonlinear stochastic heat equation 
with drift}

\author{Raluca M. Balan\footnote{Corresponding author. University of Ottawa, Department of Mathematics and Statistics, 150 Louis Pasteur Private, Ottawa, ON, K1N 6N5, Canada. E-mail address: rbalan@uottawa.ca.} \footnote{Research supported by a grant from the Natural Sciences and Engineering Research Council of Canada.}
\and
Michael Salins\footnote{Boston University, Department of Mathematics and Statistics, 665 Commonwealth Avenue, Boston, MA, 02215 , USA. E-mail address: msalins@bu.edu.}
\footnote{Research supported by Simons Foundation grant \#962543.}
}

\date{December 16, 2025}
\maketitle

\begin{abstract}
\noindent In this article, we prove the Quantitative Central Limit Theorem (QCLT) for the spatial average of the solution of the nonlinear stochastic heat equation with constant initial condition, driven by space-time Gaussian white noise in dimension 1. The novelty is that the equation contains a drift term. We assume that the drift and diffusion coefficients are twice differentiable with bounded first and second order derivatives. For the proof, we use Malliavin calculus, and the second-order Poincar\'e inequality due to \cite{vidotto20}. To estimate the moment of the second Malliavin derivative of the solution, we develop a novel estimate for the product of two heat kernels, which is of independent interest. Finally, we provide the functional result corresponding to this CLT.
\end{abstract}


\vspace{1mm}



\section{Introduction}

In the recent years, there has been a lot of interest in examining the asymptotic behaviour
as $R \to \infty$ of the {\em spatial average}:
\[
F_R(t)=\int_{|x|\leq R}\big( u(t,x)-\bE[u(t,x)]\big)dx
\]
associated to the mild solution $\{u(t,x);t\geq 0,x\in\bR^d\}$ of a stochastic partial differential equation (SPDE), interpreted as a random field, in the Walsh-Dalang sense \cite{walsh86,dalang99}. 
When the noise is spatially homogeneous and the initial condition is constant, for any fixed $t>0$ the solution $\{u(t,x)\}_{x\in \bR}$ is strictly stationary. With little additional effort, using the latest developments based Malliavin calculus (see e.g. \cite{CKNP21}), it can also be proved that $\{u(t,x)\}_{x\in \bR}$ is spatially ergodic. Hence, by
the mean ergodic theorem, the following law of large numbers holds: 
\[
\frac{1}{R}F_R(t) \to 0 \quad \mbox{a.s. and in $L^2(\Omega)$}, \ \mbox{as $R \to \infty$}.
\]

A natural question is whether the family $\{F_R(t)\}_{R>0}$ also satisfies a central limit theorem. To investigate this, a novel approach was proposed in \cite{HNV20}, combining Malliavin calculus with Stein's method for normal approximation. While this methodology was extensively developed in the monograph \cite{NP12}, its first application to SPDEs appeared in the seminal work \cite{HNV20}.

The original method of \cite{HNV20} was developed for the {\em stochastic heat equation} (SHE):
\begin{equation}
\label{SHE1}
\frac{\partial u}{\partial t} (t,x)=\frac{1}{2}\frac{\partial^2 u}{\partial x^2} (t,x) +\sigma\big(u(t,x)\big) \dot{W}(t,x), \quad t>0,x\in \bR
\end{equation}
with initial condition $u(0,x)=1$, where $\dot{W}$ is the space-time Gaussian white noise and $\sigma$ is a globally Lipschitz function. We recall that the solution of \eqref{SHE1} satisfies:
\[
u(t,x)=1+\int_0^t \int_{\bR}G_{t-s}(x-y)\sigma \big(u(s,y)\big)W(ds,dy),
\]
where the stochastic integral is interpreted in the It\^o sense, and $G$ is the heat kernel:
\[
G_t(x)=\frac{1}{\sqrt{2\pi t}} \exp\left(-\frac{|x|^2}{2t}\right), \quad t>0,x\in \bR.
\]

The main result of \cite{HNV20}, the so-called {\em Quantitative Central Limit Theorem} (QCLT) states that for any $t>0$,
\begin{equation}
\label{QCLT1} 
d_{TV}\left( \frac{F_R(t)}{\sigma_R(t)},Z\right)\les R^{-1/2},
\end{equation}
where $d_{TV}$ is the total variation distance, $\sigma_R^2(t)=\bE[F_R^2(t)]$ and $Z$ is a standard normal random variable. Moreover, the authors of \cite{HNV20} showed that $\sigma_R^2(t) \sim R$. Here, we write $a_R \les b_R$, if $a_R \leq Cb_R$ for all $R>0$, and $a_R \sim b_R$ if $\lim_{R\to \infty}a_R/b_R \in (0,\infty)$.
We recall that $d_{TV}$ is one of the metrics used for proving convergence in distribution, and is defined as follows: for any random variables $X$ and $Y$, 
\[
d_{TV}(X,Y)=\sup_{B \in \cB(\bR)}\big|\bP(X \in B)-\bP(Y \in B)\big|,
\]
where $\cB(\bR)$ is the class of Borel subsets of $\bR$.

\medskip

Subsequently, the QCLT problem has been studied for other models, possibly in higher dimensions. We recall the most important contributions below. 

\medskip

The case of (SHE) with Gausssian noise, which is white in time and colored/rough in space was studied in \cite{HNVZ20}, respectively \cite{NXZ22}. The QCLT problem for (SHE) with $\sigma(u)=u$, also called the {\em parabolic Anderson model} (pAm), was studied in \cite{NXZ22,NZ20} for the colored noise in space, and in \cite{NSZ20,NXZ22} for the rough noise in space. The QCLT problem for the {\em stochastic wave equation} (SWE):
\[
\frac{\partial^2 u}{\partial t^2} (t,x)=\Delta u(t,x) +\sigma\big(u(t,x)\big) \dot{W}(t,x), \quad t>0,x\in \bR^d
\]
with initial condition $u(0,x)=1$, $\frac{\partial u}{\partial t} (0,x)=0$, and a globally Lipschitz diffusion coefficient 
$\sigma$, driven by Gaussian noise that is white in time, was investigated in \cite{DNZ, BNZ, NZ22} for the colored noise in space in dimensions 
$d \leq 2$, when the spatial correlation of the noise is given by the Riesz kernel, or an integrable function. These results were extended to dimensions $d=3$, respectively $d\geq 4$ in \cite{ebina24,ebina25a}, for the Riesz kernel.  For the colored noise in time, the QCLT problem for (SWE) with $\sigma(u)=u$, also called the {\em hyperbolic Anderson model} (hAm), was treated in \cite{BNQZ22,BHWXY25} for the colored/rough noise in space, respectively. Finally, the QCLT for time-independent Gaussian noise was analyzed in \cite{BY23} for (pAm), and  in \cite{BY22, BY24} for (hAm). 
These results can be summarized as follows: (see the table in \cite{BHWXY25})
\begin{description}
\item[(a)] If the spatial covariance of the noise is given by the Riesz kernel $\gamma(x)=|x|^{-\beta}$ for $\beta \in (0,d)$, then the rate of convergence in QCLT is $R^{-\beta/2}$ and $\sigma_R(t) \sim R^{d-\beta/2}$. (The case $d=\beta=1$ corresponds to the space-time Gaussian white noise.)

\item[(b)] If the spatial covariance of the noise is given by a function $\gamma \in L^1(\bR^d)$, 
then the rate of convergence in QCLT is $R^{-d/2}$ and $\sigma_R(t) \sim R^{d/2}$.

\item[(c)] If $d=1$ and the noise is rough in space (i.e. it behaves in space like the fractional Brownian motion with Hurst index $H<1/2$), then the rate of convergence in QCLT is $R^{-1/2}$ and $\sigma_R(t) \sim R^{1/2}$.
\end{description}

These results are valid for both (SHE) and (SWE), regardless of whether the noise is white in time, colored in time, or time-independent. This means the that the rate of convergence in QCLT and the order of magnitude of $\sigma_R(t)$ are sensitive with respect to the spatial structure of the noise, but are robust with respect to the differential operator of the equation, and the time structure of the noise.

\medskip

The main estimate of \cite{HNV20} is based on the {\em Stein-Malliavin bound}: for any random variable $F \in \bD^{1,2}$ such that $\bE[F^2]=\sigma^2>0$ and $F=\delta(v)$ for some $v \in {\rm Dom}(\delta)$, 
\begin{equation}
\label{SM}
d_{TV}\left(\frac{F}{\sigma},Z\right) \leq \frac{2}{\sigma^2} \bE|\sigma^2 -\langle DF,v\rangle_{\cH}| \leq {\frac{2}{\sigma^2}} \sqrt{{\rm Var}\big(\langle DF,v\rangle_{\cH}\big)}.
\end{equation}
where $\cH=L^2(\bR_{+}\times \bR)$, $(D,\delta)$ are the Malliavin derivative, respectively the Skorohod integral with respect to $W$, and $\bD^{1,2},{\rm Dom}(\delta)$ are the respective domains of these operators. We refer the reader to Section \ref{section-Malliavin} below for their precise definitions. The second inequality in \eqref{SM} is due to Cauchy-Schwarz inequality, since by the duality between $D$ and $\delta$,
\begin{equation}
\label{var}
\bE[\langle DF, v\rangle_{\cH}]=\bE[F \delta(v)]=\sigma^2.
\end{equation} 

The insight of \cite{HNV20} was to represent 
\begin{equation}
\label{FR-rep}
F_R(t)=\delta(v_{R,t}) \quad \mbox{with} \quad
v_{R,t}(s,y)=\sigma(u(s,y))\int_{|x|<R}G_{t-s}(x-y) dx,
\end{equation}
and then bound the variance in \eqref{SM} using the Clark-Ocone formula (given by \eqref{CK} below) and the key estimate for the Malliavin derivative of the solution: for $p\geq 2$, $0<r<t\leq T$,
\begin{equation}
\label{key-intro}
\|D_{r,z}u(t,x)\|_p \les G_{t-r}(x-z),
\end{equation}
where $\|\cdot\|_p$ is the $L^p(\Omega)$-norm. Representation \eqref{FR-rep} was subsequently used in all other references dedicated to the QCLT problem, when the Gaussian noise is {\em white in time}, and possibly colored in space, replacing $\cH$ by the underlying Hilbert space of the noise.

\medskip

When the noise is {\em colored in time}, the existence of the solution is only known in the case $\sigma(u)=u$. In this case, the method based on \eqref{FR-rep} and Clark-Ocone formula does not yield a suitable estimate. A new idea was introduced in \cite{BNQZ22}, which works for a variety of models, including the space-time Gaussian white noise, and can be summarized as follows:
\begin{description}
\item[(i)] use a basic fact from Malliavin calculus which says that any centered random variable $F\in L^2(\Omega)$ can be written as $F=\delta(v)$ with $v=-DL^{-1}F$, where $L^{-1}$ is the pseudo-inverse of the OU generator, 
     and represent $F_R(t)$ in this way, instead of \eqref{FR-rep};

\item[(ii)] to estimate the variance in \eqref{SM}, use the recent sharp bound for ${\rm Var}\big(\langle DF,-DL^{-1}F\rangle_{\cH}\big)$ obtained by Vidotto in \cite{vidotto20} for the space-time Gaussian white noise, called the {\em second-order Gaussian Poincar\'e inequality}, or an adaptation of this bound for the noise used in the equation, as for instance, Proposition 1.8 of \cite{BNQZ22} for the colored noise in space and time, or Proposition 2.4 of \cite{NXZ22} for the rough noise in space, colored in time;
    
\item[(iii)] depending on the form of the bound in {\bf (ii)}, develop a suitable estimate for the second-order Malliavin derivative of the solution, and use it in conjunction with the estimate \eqref{key-intro}. For instance, for the colored noise in space and time, the estimate given by Theorem 1.3 of \cite{BNQZ22} states that for $p\geq 2$ and $0<\theta<r<t\leq T$,
\begin{equation}
\label{key-D2-intro}
\|D_{(\theta,w),(r,z)}^2 u(t,x)\|_p \les G_{t-r}(x-z)G_{r-\theta}(z-w),
\end{equation} 

\end{description}

We conclude the literature review by mentioning some other related contributions.
The CLT for spatial averages of the form:
\[
\int_{[0,R]^d} \big(g(u(t,x))-\bE[g(u(t,x))]\big) dx
\]
was proved in \cite{CKNP23} for (SHE) in dimension $d\geq 1$, driven by Gaussian white noise, colored in time, when $g$ is a globally Lipschitz function. The QCLT for (hAm) in dimension $d=1$, driven by a finite-variance L\'evy noise was proved in \cite{BZ24}, using the recent advances Malliavin calculus on the Poisson space due to \cite{trauthwein25}.  
A large deviations principle for the solutions to (SHE) and (SWE) driven by a Gaussian noise which is white in time and colored in space has been recently obtained in \cite{ebina25b}.

Assuming that $\bE[u(t,x)]=1$, and letting $R=\e^{-1}$ and $g(x)=1_{\{|x|\leq 1\}}$, we see that
\[
\frac{F_R(t)}{R^{d/2}}=\frac{1}{\e^{d/2}}\int_{\bR^d} \left[u\left(t,\frac{x}{\e}\right)-1\right]
g(x)dx,
\]
which converges to a normal distribution as $\e \to 0$,
in case {\bf (b)} mentioned above. A different fluctuation result was obtained in \cite{gu-li20}, which involves also time rescaling: if $u$ is the solution of (SHE) in dimension $d\geq 3$ with initial condition 1, driven by Gaussian noise white in time and colored in space with spatial covariance kernel $\gamma =\phi *\phi$ for some $\phi \in C_c^{\infty}(\bR^d)$, then 
\[
\frac{1}{\e^{d/2-1}}\int_{\bR^d} \left[u\left(\frac{t}{\e^2},\frac{x}{\e}\right)-1\right]
g(x)dx \Longrightarrow  \int_{\bR^d} \cU(t,x)g(x)dx \quad \mbox{as $\e \to 0$},
\]
for any $g \in C_c^{\infty}(\bR^d)$, where $\cU$ is the solution of the Edwards-Wilkinson equation.

\bigskip

All the studies cited above deal with SPDEs without a drift term. 
The goal of the present article is to establish the QCLT for an equation \underline{with drift}, which to the best of our knowledge, has not been treated before. For simplicity, we consider (SHE) driven by space-time Gaussian white noise in dimension $d=1$, and postpone the study of the same problem for higher dimensions with colored or rough noise for future work.

\medskip

More precisely, we consider the equation:
\begin{align}
\label{SHE}
	\begin{cases}
		\dfrac{\partial u}{\partial t} (t,x)
		=  \dfrac{1}{2}\dfrac{\partial^2 u}{\partial x^2} (t,x)+ b\big(u(t,x)\big)+\s\big(u(t,x)\big) \dot{W}(t,x), \
		t>0, \ x \in \bR, \\
		u(0,x) = 1, 
	\end{cases}
\end{align}
where $W$ is a space-time Gaussian white noise, and $b,\sigma$ are globally Lipschitz.

\medskip

We recall that $W=\{W(\varphi);\varphi \in \cH\}$ is a zero-mean Gaussian process indexed by the Hilbert space $\cH=L^2(\bR_{+}\times \bR)$, defined on a complete probability space $(\Omega,\cF,\bP)$, and with covariance
$\bE[W(\varphi)W(\psi)]=\langle \varphi,\psi \rangle_{\cH}$. 
We let $(\cF_t)_{t\geq 0}$ be the filtration induced by $W$:
\[
\cF_t=\sigma\{ W(1_{[0,s]\times A}); s \in [0,t],A \in \cB_b(\bR)\} \vee \cN,
\]
where $\cB_b(\bR)$ is the class of bounded Borel subsets of $\bR$, and $\cN$ is the class of $\bP$-negligible sets.
A process $\{X(t,x);t\geq 0,x\in \bR\}$ is {\em predictable} if it $\cP$-measurable, where
$\cP$ is the
$\sigma$-field generated by elementary processes of the form 
$X(t,x)=Y1_{(a,b]}(t) 1_{A}(x)$, where $0<a<b$, $A\in \cB_b(\bR)$ and $Y$ is $\cF_a$-measurable. 

\medskip
For any predictable process $X$ such that $\bE\int_0^T \int_{\bR} |X(t,x)|^2 dtdx<\infty$ for any $T>0$, we can define the {\em It\^o integral} of $X$ with respect to $W$. This integral is an isometry:
\[
\bE\left|\int_0^T \int_{\bR} X(t,x) W(dt,dx)\right|^2 =\bE \int_0^T \int_{\bR}|X(t,x)|^2 dtdx.
\]
Moreover, $M_t=\int_0^t \int_{\bR} X(s,x)W(ds,dx)$ is a continuous martingale with respect to $(\cF_t)_{t\geq 0}$, with predictable quadratic variation $\langle M \rangle_t=\int_0^t \int_{\bR}|X(s,x)|^2 dsdx$. To estimate the moments of $M_t$, we will use the {\em Burkholder-Davis-Gundy (BDG) inequality}: for any $p\geq 2$,
\begin{equation}
\label{BDG}
\|M_t\|_p^2 \leq z_p^2 \, \|\langle M\rangle_t\|_{p/2},
\end{equation}
where $z_p>0$ is a constant depending on $p$.

\begin{definition}
{\rm 
We say that $u=\{u(t,x);t\geq 0,x\in \bR\}$ is a (mild) {\em solution} of \eqref{SHE} if $u$ is predictable and satisfies the integral equation:
\begin{equation}
\label{eq}
u(t,x)=1+\int_0^t \int_{\bR}G_{t-s}(x-y)b\big(u(s,y)\big)dyds+\int_0^t \int_{\bR}G_{t-s}(x-y)\s\big(u(s,y)\big)W(ds,dy).
\end{equation}
}
\end{definition}

Following \cite{walsh86}, we know that equation \eqref{SHE} has a unique solution $u$. Moreover, it was shown in \cite{dalang99} that the process $\{u(t,x)\}_{x\in \bR}$ is strictly stationary, and 
\begin{equation}
\label{m-u-bded}
K_{T,p}:=\sup_{(t,x) \in [0,T] \times \bR}\|u(t,x)\|_p<\infty \quad \mbox{for any $p\geq 2$ and $T>0$}.
\end{equation}

We introduce the following assumptions.

\begin{assumption}
\label{assumpt1}
{\rm
$b,\sigma\in C^1(\bR)$ and $b',\sigma'$ are bounded.
}
\end{assumption}

\begin{assumption}
\label{assumpt2}
{\rm
$b,\sigma\in C^2(\bR)$ and $b',b'',\sigma',\sigma''$ are bounded.
}
\end{assumption}

We are now ready to state the main results of this article.

\begin{theorem}
\label{main1}
If Assumption \ref{assumpt1} holds, then:\\
(i) the process $\{u(t,x)\}_{x \in \bR}$ is ergodic, for any $t>0$ fixed;\\
(ii) for any $t,s>0$,
\[
\lim_{R \to \infty}\frac{1}{R}\bE\big[F_R(t) F_R(s) \big] = 2\int_{\bR}\rho_{t,s}(x) dx=:\Sigma_{t,s} \ \mbox{is finite},
\]
where $\rho_{t,s}(x-y)=\bE[u(t,x)u(s,y)]$. In particular, $\sigma_R^2(t) \sim R$. 
\end{theorem}

\begin{theorem}
\label{main-th}
If Assumption \ref{assumpt2} holds, then:\\
(i) (QCLT)
for any $t>0$ and $R>0$,
\begin{equation}
\label{QCLT}
d_{TV}\left(\frac{F_R(t)}{\s_R(t)},Z \right) \leq C_t R^{-1/2},
\end{equation}
where $Z \sim N(0,1)$ and $C_t>0$ is a constant depending on $t$;\\
(ii) (FCLT) 
for any $R>0$, the process $F_R(\cdot)=\{F_R(t)\}_{t\geq 0}$ has a {$\gamma$-H\"older continuous modification of order $\gamma\in (0,\frac{1}{2})$}, which we denote also $F_R(\cdot)$, and for which
\[
\frac{1}{\sqrt{R}}F_R(\cdot) \stackrel{d}{\to} \cG(\cdot) \quad \mbox{in $C([0,\infty))$},
\]
where $\cG(\cdot)=\{\cG(t)\}_{t\geq 0}$ is a zero-mean Gaussian process with
$\bE[\cG(t) \cG(s)]=\Sigma_{t,s}$. Here $C([0,\infty))$ is the space of continuous functions on $[0,\infty)$, equipped with the topology of uniform convergence on compact sets.
\end{theorem}

The proof of Theorem \ref{main-th} uses the steps {\bf (i)}-{\bf (iii)} outlined above, based on the second-order Poincar\'e inequality of Vidotto \cite{vidotto20}, which we include in Proposition \ref{prop18-1} below. We will show that the key estimate \eqref{key-intro} holds, 
and in addition, 
\begin{equation}
\label{key2-intro}
\left\| D^2_{(r,z),(\theta,w)}u(t,x)\right\|_p \les \left( 1+ |r-\theta|^{-\frac{1}{2}}\right)G_{8T}(z-w) \Big(G_{t-r}(x-z) + G_{t-\theta}(x-w)\Big).
\end{equation}

To prove inequality \eqref{key2-intro}, we use a novel technique for bounding the product of two heat kernels, which can be of independent interest; see Lemma \ref{lem:kernel-product} below. In the linear setting where $\sigma''(u)\equiv 0$ and $b''(u)\equiv 0$, the tighter bound \eqref{key-D2-intro} is valid. We do not believe that \eqref{key-D2-intro} is valid in the nonlinear setting. {The equation satisfied by the second Malliavin derivative of the solution $u$} contains the Lebesgue integral
\[\int_0^t \int_{\mathbb{R}} G_{t-s}(x-y)b''(u(s,y))D_{r,z}u(s,y)D_{\theta,w}u(s,y)dyds\]
and the stochastic integral
\[\int_0^t \int_{\mathbb{R}} G_{t-s}(x-y)\sigma''(u(s,y))D_{r,z}u(s,y)D_{\theta,w}u(s,y)W(dyds).\]
These {integrals} are difficult to bound in the nonlinear setting. 
Using \eqref{key-intro}, Minkowski's inequality, and BDG inequality for the stochastic integral, we can bound the moments of these integrals by multiples of
\[\int_0^t \int_{\mathbb{R}} G_{t-s}(x-y) G_{s-r}(y-z)G_{s-\theta}(y-w)dyds\]
and
\[\int_0^t \int_{\mathbb{R}} G^2_{t-s}(x-y) G^2_{s-r}(y-z)G^2_{s-\theta}(y-w)dyds.\]
{The right-hand side of \eqref{key2-intro} is larger than product $G_{t-r}(x-z)G_{t-\theta}(x-w)$ (see Lemma \ref{lem:kernel-product} below), and is a solution to the heat equation in the $(t,x)$ variables. Additionally, this expression decays quickly in $z$ and $w$, which is useful for the proof of QCLT.}

We denote
\begin{equation}
\label{def-Lbs}
L_b := \|b'\|_{\infty} \quad \mbox{and} \quad L_\sigma:= \|\sigma'\|_{\infty}.
\end{equation}
and $\|f\|_{\infty}=\sup_{x\in \bR}|f(x)|$ for any function $f:\bR \to \bR$.
Then, for any $x,y \in \bR$.
\begin{equation}
\label{Lip}
|b(x)-b(y)|\leq L_{b}|x-y| \quad \mbox{and} \quad
|\s(x)-\s(y)|\leq L_{\s}|x-y|.
\end{equation}

The constants appearing in \eqref{key-intro} and \eqref{key2-intro} depend only on $(T,p,L_{b},L_{\sigma},\|b''\|_{\infty},\|\s''\|_{\infty})$.

\medskip

This article is organized as follows. In Section \ref{section-prelim}, we present some basic facts from Malliavin calculus, and simple properties of the heat kernel.
In Section \ref{section-Mal}, we prove the key estimates \eqref{key-intro} and \eqref{key2-intro} for the first and second Malliavin derivatives of the solution.
In Section \ref{section-first-order}, we give the proof of Theorem \ref{main1}.
The proof of Theorem \ref{main-th} is presented in
Section \ref{section-QCLT} (part (i)), and Section \ref{section-FCLT} (part (ii)). The appendix contains some auxiliary results.

\section{Preliminaries}
\label{section-prelim}

In this section, we introduce some preliminary material which is used in this article.

\subsection{Malliavin Calculus}
\label{section-Malliavin}

In this section, we present some basic facts from Malliavin calculus with respect to the isonormal Gaussian process $W=\{W(\varphi)\}_{\varphi \in \cH}$. We refer the reader to \cite{nualart06} for more details.

\medskip

$\bullet$ {\bf Chaos expansion.}
Every random variable $F\in L^2(\Omega)$
which is measurable with respect to $W$ has the Wiener chaos expansion:
\begin{equation}
\label{F-chaos}
F=\bE(F)+\sum_{n \geq 1}I_n(f_n) \quad \mbox{for some} \quad f_n \in
\cH^{\otimes n},
\end{equation}
where $\cH^{\otimes n}=L^2\big((\bR_{+} \times \bR)^n\big)$, $I_n:\cH^{\otimes n} \to \cH_n$
is the multiple Wiener integral with respect to $W$, and $\cH_n$ is the $n$-th chaos space.
We denote by $J_n$ the projection on the $n$-th Wiener chaos.
By the orthogonality of the Wiener chaos spaces, if $F$ has the chaos expansion \eqref{F-chaos}, then
$$E|F|^2=[\bE(F)]^2+ \sum_{n \geq 1}E|I_n(f_n)|^2=[\bE(F)]^2+\sum_{n \geq 1}n! \,
\|\widetilde{f}_n\|_{\cH^{\otimes n}}^{2},$$
where $\widetilde{f}_n$ is the symmetrization of $f_n$. 

\medskip

$\bullet$ {\bf Malliavin derivative.}
Let $\cS$ be the class of ``smooth'' random variables:
\begin{equation}
\label{form-F}F=f(W(\varphi_1),\ldots, W(\varphi_n)),
\end{equation} where $f \in C_{b}^{\infty}(\bR^n)$, $\varphi_i \in \cH$,
$n \geq 1$, and
$C_b^{\infty}(\bR^n)$ is the class of bounded $C^{\infty}$-functions
on $\bR^n$, whose partial derivatives of all orders are bounded. The
{\em Malliavin derivative} of $F$ of the form (\ref{form-F}) is the
$\cH$-valued random variable given by:
$$D_{t,x}F:=\sum_{i=1}^{n}\frac{\partial f}{\partial x_i}(W(\varphi_1),\ldots,
W(\varphi_n))\varphi_i(t,x), \quad \mbox{for $t>0$ and $x\in \bR$}.$$ We endow $\cS$ with the semi-norm
$\|F\|_{\bD^{1,2}}=(\bE|F|^2+(E\|D F \|_{\cH}^{2})^{1/2}$. The
operator $D$ can be extended to the space $\bD^{1,2}$, the
completion of $\cS$ with respect to $\|\cdot \|_{\bD^{1,2}}$.

\medskip

Proposition 1.2.3 (on page 28) of \cite{nualart06} states that the Malliavin derivative satisfies the following chain rule:
for any $F \in \bD^{1,2}$ and $\varphi \in C^1(\bR)$ with $\varphi'$ bounded,
\begin{equation}
\label{chain}
D\varphi(F)=\varphi'(F)DF.
\end{equation}

Similarly to Corollary 1.2.1 (on page 34) of \cite{nualart06}, if $F\in \bD^{1,2}$ is $\cF_t$-measurable, then 
\begin{equation}
\label{Mal-zero}
D_{r,z}F=0 \quad \mbox{for all $r>t$ and $z \in \bR$}.
\end{equation}

We recall the {\em Poincar\'e inequality}: for any $F \in \bD^{1,2}$
\begin{equation}
\label{Poincare}
{\rm Var}(F) \leq \bE\| DF\|_{\cH}^2=\int_0^{\infty}\int_{\bR}\bE|D_{t,x}F|^2 dtdx,
\end{equation}
and the {\em Clark-Ocone formula}: for any $F,G \in \bD^{1,2}$, 
\begin{equation}
\label{CK}
F=\bE(F)+\int_0^{\infty}\int_{\bR} \bE[D_{t,x} F|\cF_t] W(dt,dx).
\end{equation}

Using the Clark-Ocone formula, 
we obtain that 
for any $F,G \in \bD^{1,2}$,
\begin{equation}
\label{Poin-cov}
\left|{\rm Cov}(F,G) \right| \leq \int_0^{\infty} \int_{\bR}\|D_{t,x}F\|_{2} \|D_{t,x}G\|_{2} dtdx.
\end{equation}

We can define the iteration of the operator $D$ such that for any $F \in \cS$, the iterated derivative $D^k F$ is an element in $L^2(\Omega;\cH^{\otimes k})$. For any $p\geq 1$ and natural number $k\geq 1$, we introduce the semi-norm on $\cS$:
\[
\|F\|_{\bD^{k,p}}=\left( \bE|F|^p+ \sum_{j=1}^k \bE\|D^j F\|_{\cH^{\otimes j}}^p \right)^{1/p}.
\]
The operator $D^k$ is extended to the space $\bD^{k,p}$, the
completion of $\cS$ with respect to $\|\cdot \|_{\bD^{k,p}}$.

Let $F \in L^2(\Omega)$ be arbitrary. For any $n \geq k$, $J_nF \in \bD^{k,2}$, $D^k (J_n F)=J_{n-k}(D^{k}F)$ and 
\begin{equation}
\label{Dk-Jn}
\bE\|D^k (J_n F)\|_{\cH^{\otimes k}}^2=n(n-1)\ldots (n-k+1) \bE|J_n F|^2.
\end{equation}
Moreover, $F \in \bD^{k,2}$ if and only if
\begin{equation}
\label{crit}
S:=\sum_{n\geq k}n(n-1)\ldots(n-k+1)\bE|J_n F|^2 =\sum_{n\geq 1}\bE\|D^k (J_n F)\|_{\cH^{\otimes k}}^2<\infty,
\end{equation}
and in this case $\bE\|D^k F\|_{\cH^{\otimes k}}^2=S$.
\medskip

$\bullet$ {\bf Skorohod integral.}
The {\em divergence operator} $\delta$ is the adjoint of
the operator $D$. The domain of $\delta$, denoted by $\mbox{Dom} 
(\delta)$, is the set of $u \in L^2(\Omega;\cH)$ such that
$$|\bE \langle DF,u \rangle_{\cH}| \leq c (\bE|F|^2)^{1/2}, \quad \forall F \in
\bD^{1,2},$$
where $c$ is a constant depending on $u$. If $u \in {\rm Dom} (\delta)$, then $\delta(u)$ is the element of $L^2(\Omega)$
characterized by the following duality relation:
\begin{equation}
\label{duality}
\bE(F \delta(u))=\bE\left(\langle DF,u \rangle_{\cH}\right), \quad
\forall F \in \bD^{1,2}.
 \end{equation}
In particular, $\bE[\delta(u)]=0$. If $u \in \mbox{Dom} (\delta)$, 
we use the notation
$\delta(u)=\int u(t,x) W(\delta t, \delta x)$, and
we say that $\delta(u)$ is the {\em Skorohod integral} of $u$. If $u$ is jointly measurable and $(\cF_t)_{t\geq 0}$-adapted (in particular predictable), then $\delta(u)$ coincides with the It\^o integral of $u$.

\medskip

Let $\bL^{1,2}$ be the class of processes $u \in L^2(\Omega;\cH)$ such that $u(t,x)\in \bD^{1,2}$ for all $(t,x)\in \bR_{+}\times \bR$, and there exists a measurable modification of
$\{D_{t,x}u(s,y);(t,x),(s,y)\}$ such that
\[
\bE \int_{(\bR_{+}\times \bR)^2} |D_{t,x}u(s,y)|^2 dtdxdsdy<\infty.
\]

Proposition 1.3.8 of \cite{nualart06} states that if $u \in \bL^{1,2}$, $\{D_{t,x}u(s,y)\}_{(s,y)} \in {\rm Dom}(\delta)$ for almost all $(t,x)$, and there exists a measurable modification of $\{\int D_{t,x}u(s,y)W(\delta s,\delta y)\}_{(t,x)}$ which is in $L^2(\Omega;\cH)$, then $\delta(u) \in \bD^{1,2}$ and the following {\em Heisenberg commutation} holds:
\begin{equation}
\label{Heisenberg}
D(\delta(u))=u+\delta(Du).
\end{equation}

\medskip

$\bullet$ {\bf Ornstein-Uhlenbeck generator.}
If $F\in L^2(\Omega)$ has the chaos expansion \eqref{F-chaos}, we define the {\em Ornstein-Uhlenbeck (OU) generator}
\begin{equation}
\label{def-L}
LF=-\sum_{n\geq 1}n I_n(f_n),
\end{equation}
whose domain is the set ${\rm Dom}(L)$ of all random variables $F \in L^2(\Omega)$ for which the series converges in $L^2(\Omega)$.
If $F\in L^2(\Omega)$ has the chaos expansion \eqref{F-chaos}, we let:
\[
L^{-1}F=-\sum_{n\geq 1}\frac{1}{n} I_n(f_n).
\]
 $L^{-1}$  is the pseudo-inverse of $L$, in the sense that
$L L^{-1} F=F-\bE(F)$ for any $F \in L^2(\Omega)$.

\medskip

By Proposition 1.4.3 of \cite{nualart06},
$F \in {\rm Dom}\ L$ if and only if $F \in \bD^{1,2}$ and $DF \in {\rm Dom} \ \delta$, and in this case, $LF=-\delta (D F)$. Applying this result to $F'=L^{-1} F \in {\rm Dom}(L)$, we conclude that for any $F \in L^2(\Omega)$ with $\bE(F)=0$, the process $v=-D L^{-1}F$ belongs to ${\rm Dom} \ \delta$
and
\begin{equation}
\label{Mal-cor}
F=\delta(-D L^{-1} F).
\end{equation}

\subsection{Properties of the heat kernel}

In this section, we include some properties of the heat kernel which are used in the sequel.

\medskip

$\bullet$ {\bf Product of squares.}
For any $t>0$ and $x \in \bR$,
\begin{equation}
\label{G2-G}
G_t^2(x)=\frac{1}{\sqrt{4\pi}} t^{-1/2}G_{t/2}(x).
\end{equation}

By \eqref{G2-G} and the semigroup property, for any $s \in [r,t]$
\begin{align}
\label{G2-G2}
\int_{\bR} G^2_{t-s}(x-y) G_{s-r}^2(y-z) dy &
= \frac{1}{\sqrt{4\pi}}  (t-s)^{-\frac{1}{2}} (s-r)^{-\frac{1}{2}} (t-r)^{\frac{1}{2}} G_{t-r}^2(x-z).
\end{align}
Therefore,
\begin{equation}
\label{G2-G2-2}
\int_r^t \int_{\bR} G^2_{t-s}(x-y) G_{s-r}^2(y-z) dy ds=\sqrt{\frac{\pi}{4}}(t-r)^{\frac{1}{2}} G_{t-r}^2(x-z),
\end{equation}
using the fact that 
\begin{equation}
\label{beta}
\int_r^t (t-s)^{-\frac{1}{2}}(s-r)^{-\frac{1}{2}}ds={\rm Beta}\left(\frac12,\frac12\right)=\pi.
\end{equation}

\medskip
$\bullet$ {\bf Monotonicty.} For any $0 <s \leq t$ and $x \in \mathbb{R}$,
\begin{equation} 
\label{eq:gaussian-bound}
G_s(x) \leq s^{-\frac{1}{2}}  t^{\frac{1}{2}} G_t(x).
\end{equation}


\section{Moment bounds for the Malliavin derivatives}
\label{section-Mal}

In this section, we prove the key estimates for the first and second Malliavin derivatives of the solution $u$, which are will be used for the proofs of Theorems \ref{main1} and \ref{main-th}. 
First, we will obtain these estimates for the Malliavin derivatives of an approximation $u_n$ of the solution. 
Then, by applying Lemma \ref{lem:weak-implies-pointwise} below, we will conclude that the same estimates hold for $Du$ and $D^2u$.

As approximating sequence, we consider
the Picard iterations defined by: $u_0(t,x)=1$ and
\begin{equation}
\label{def-Picard}
u_{n+1}(t,x)=1+\int_0^t \int_{\bR}G_{t-s}(x-y)b(u_n(s,y))dyds+\int_0^t \int_{\bR}G_{t-s}(x-y)\s(u_n(s,y))W(ds,dy).
\end{equation}

{
From Theorem 13 of \cite{dalang99}, it is known that for any $p\geq 2$ and $T>0$,
\begin{equation}
\label{un-conv-u}
\sup_{t \in [0,T]}\sup_{x\in \bR}\|u_n(t,x)-u(t,x)\|_p \to 0 \quad \mbox{as $n \to \infty$}.
\end{equation}
Moreover, by Lemma 7.3 of \cite{sanz05}, for any integer $k\geq 1$ and $p\geq 2$, $u_n \in \bD^{k,p}$  for any $n$, and 
\[
\sup_{n\geq 1}\sup_{(t,x)\in [0,T]\times \bR}\bE \|D^k u_n(t,x)\|_{\cH^{\otimes k}}^p<\infty, \quad \mbox{for any $p\geq 2,T>0$}.
\]
}

Note that $b$ and $\sigma$ satisfy the following linear growth property: for any $x \in \bR$,
\begin{equation}
\label{lin-grow}
|b(x)|\leq D_{b} (|x|+1) \quad \mbox{and} \quad  |\s(x)|\leq D_{\sigma} (|x|+1),
\end{equation}
where $D_{b}=\max\{L_{b},|b(0)|\}$ and $D_{\sigma}=\max\{L_{\sigma},|\sigma(0)|\}$.
It follows that
{
\begin{equation}
\label{s-u-bded}
\sup_{(t,x) \in [0,T] \times \bR}\|\s\big(u(t,x)\big)\|_p\leq D_{\sigma} (K_{T,p}+1)=:K_{T,p,D_{\sigma}},
\end{equation}
where $K_{T,p}$ is given by \eqref{m-u-bded}.
}

\medskip
 
We will use the following result (see Lemma 1.2.3 of \cite{nualart06}).
  
\begin{lemma}
\label{lem:nualart123}
Let $(F_n)_{n\geq 1}$ be a sequence in $\bD^{1,2}$ such that $F_n \to F$ in $L^2(\Omega)$, and
\[
\sup_{n\geq 1} \bE \|DF_n\|_{\cH}^2<\infty.
\]
Then $F \in \bD^{1,2}$ and the sequence $(DF_n)_{n\geq 1}$ converges to $DF$ in the weak topology of $L^2(\Omega;\cH)$.
  \end{lemma}

\begin{theorem} 
\label{thm:1st-deriv-prelim}
If Assumption \ref{assumpt1} holds, then 
for any $p\geq 2$, $0<r<t \leq T$ and $x,z \in \bR$,
\begin{equation}
\label{key}
\|D_{r,z}u(t,x)\|_p \leq \cC_{T,p}^{(1)} G_{t-r}(x-z),
\end{equation}
where {$\cC_{T,p}^{(1)}$ is a constant that depends on $(p,T,L_b,L_\sigma)$. Here $(L_b,L_{\s})$ are given by \eqref{def-Lbs}}.
\end{theorem}

\begin{proof} We use a similar argument as in the proof of Theorem 1.6.(a) of \cite{CH23}, except that these authors used this argument for the solution $u$ itself, instead of $Du_n$. 

Using Heisenberg commutation rule \eqref{Heisenberg} followed by the chain rule \eqref{chain}, we infer that the sequence $(Du_n)_{n\geq 0}$ satisfies the following recurrence relation:
\begin{align}
\nonumber
D_{r,z}u_{n+1}(t,x)&=G_{t-r}(x-z)\s\big(u_n(r,z)\big)+\int_r^t \int_{\bR}G_{t-s}(x-y) b'\big(u_n(s,y)\big)D_{r,z}u_n(s,y) dyds+\\
\label{rec-Dun}
&\quad \int_r^t \int_{\bR} G_{t-s}(x-y)\s'\big(u_n(s,y)\big) D_{r,z}u_n(s,y) W(ds,dy), 
\end{align}
for any $n\geq 0$. Note that $D_{r,z}u_{0}(t,x)=0$ since $u_0(t,x)=1$.

By Minkowski's inequality and BDG inequality \eqref{BDG}, for any $0<r<t\leq T$, $x,z \in \bR$ and $n\geq 0$,
\begin{align*}
  		\|D_{r,z}u_{n+1}(t,x)\|_p^2 & \leq 3 \left\{K_{p,T,D_{\s}}^2 G_{t-r}^2(x-z)+  L_b^2 \left(\int_r^t \int_{\bR} G_{t-s}(x-y)\|D_{r,z}u_n(s,y)\|_pdyds\right)^2 \right.\\
  		&+\left. z_p^2 L_\sigma^2\int_r^t \int_{\bR} G^2_{t-s}(x-y)\|D_{r,z}u_n(s,y)\|_p^2dyds \right\},
  	\end{align*}
{where $K_{p,T,\sigma}$ is given by \eqref{s-u-bded}.}
For fixed $(r,z)\in \bR_{+}\times \bR$,
the function $f_n(t,x) = \|D_{r,z}u_n(t,x)\|_p$ satisfies the assumptions of Lemma \ref{lemA2-HNV} {on $I=(r,T]$}, with $\rho(t,x) = G_{t-r}(x-z)$, $A=K_{T,p,D_{\s}}$
and $C_t=C_{p,L_b,L_{\sigma}}$ {given by
\begin{equation}
\label{def-CpL}
C_{p,L_b,L_{\sigma}}=\max(1,L_b^2,z_p^2 L_{\sigma}^2).
\end{equation}
Here $z_p$ is the constant from the BDG inequality \eqref{BDG}.
Therefore,there exists a constant $\cC_{T,p}^{(1)}>0$ depending on $(T,p,L_b,L_{\s})$,} such that for any $n\geq 1$, $0<r < t\leq T$ and $x,z \in \bR$,
\begin{equation} 
\label{key-un}
\|D_{r,z}u_n(t,x)\|_p \leq \cC_{T,p}^{(1)} G_{t-r}(x-z).
\end{equation}

We apply Lemma \ref{lem:nualart123} to $F_n=u_n(t,x)$ and $F=u(t,x)$ for fixed $(t,x)$.
It follows that $u(t,x)\in \bD^{1,2}$ and $\{Du_n(t,x)\}_{n\geq 1}$ converges to $Du(t,x)$ in the weak topology $L^2(\Omega;\cH)$.
Then, the conclusion follows by Lemma \ref{lem:weak-implies-pointwise}  and relation \eqref{key-un}.
\end{proof}
  
Next, we will provide an estimate for the $p$-th moment of the second Malliavin derivative $D^2 u$.
For this we will need to handle products of two heat kernels $G$.
Let us first remark that from Lemma A.4 of \cite{chen-dalang15}, we know that for any $t,s>0$ and $x,y \in \mathbb{R}$, 
		\begin{equation} 
			\label{mult-Gaussian}
			G_t(x)G_s(y) = G_{t+s}(x-y)G_{\frac{ts}{t+s}}\left(\frac{sx+ty}{t+s}\right).
		\end{equation}
		As an application of \eqref{mult-Gaussian}, we obtain that for any $r<s<t$ and $x,y,z \in \bR$, we have:
		\begin{equation}
			\label{GG-id}
			G_{t-s}(x-y) G_{s-r}(y-z)=G_{t-r}(x-z) G_{\frac{(t-s)(s-r)}{t-r}}\left(y-\frac{(t-s)z
			}{t-r} -\frac{(s-r)x}{t-r}\right).
		\end{equation}
The exact expression \eqref{GG-id} is very difficult to integrate in the $s$ variable, and will not be used below.


The next technical lemma provides a different upper bound for products of Gaussian densities $G_{t-r}(x-z)G_{t-\theta}(x-w)$. 
This bound is less accurate than \eqref{GG-id}, 
but is more convenient for calculations because the it includes only heat kernels evaluated at $x-z$, $x-w$ and $w-z$.

 Importantly, the upper bound in \eqref{eq:kernel-product}, as a function of $(t,x)$, is a solution to the heat equation. This enables us to prove that a similar expression is an upper bound for the $p$-th moment of $D^2_{(r,z),(\theta,w)}u(t,x)$. Additionally, the $G_{8T}(z-w)$ factor is important for proving the QCLT because it shows that these products decay quickly if $|w-z|$ is large.

\begin{lemma} \label{lem:kernel-product}
For any $0<\theta<r<t\leq T$ and $x,w,z \in \mathbb{R}$,
\begin{align} \label{eq:kernel-product}
	&G_{t-r}(x-z)G_{t-\theta}(x-w) \nonumber\\
	&\leq 8T^{\frac{1}{2}}(1 + e^{\frac{1}{16T}})(1 + (r-\theta)^{-\frac{1}{2}})G_{8T}(z-w)\Big(G_{t-r}(x-z) + G_{t-\theta}(x-w) \Big) .
\end{align}
\end{lemma}

\begin{proof}
We split the analysis of this product into two cases: if $|w-z|\leq 1$ or if $|w-z|>1$. 

\medskip

{\em Case 1.} Assume that $|z-w| \leq 1$. In this case, we bound the factor $G_{t-\theta}(x-w)$. More precisely, using the fact that
\[
1 \leq e^{\frac{1}{16T}} \cdot e^{-\frac{|z-w|^2}{16T}} =  e^{\frac{1}{16T}} \cdot (16 \pi T )^{\frac{1}{2}}G_{8T}(z-w),
\]
we can bound for any $0<\theta<r<t\leq T$ and $x,w \in \mathbb{R}$,
\[
G_{t-\theta}(x-w) \leq (2\pi)^{-\frac{1}{2}}(t-\theta)^{-\frac{1}{2}} \leq (2\pi)^{-\frac{1}{2}}(r-\theta)^{-\frac{1}{2}} \leq (r-\theta)^{-\frac{1}{2}} (8T)^{\frac{1}{2}} e^{\frac{1}{16T}}  G_{8T}(z-w).
\]
{Multiplying this inequality by $G_{t-r}(x-z)$,} we bound the product as follows:
\begin{equation}
	\label{eq:G-prod-w-z-small}
	G_{t-r}(x-z)G_{t-\theta}(x-w) \leq (r-\theta)^{-\frac{1}{2}} (8T)^{\frac{1}{2}} e^{\frac{1}{16T}}  G_{8T}(z-w)G_{t-r}(x-z)
\end{equation}

{\em Case 2.} Assume that $|z-w|>1$.
By the triangle inequality, for any $x,y,z \in \mathbb{R}$
\[|z-w| \leq |x-z| + |x-w|.\]
This implies that either
$|x-z|\geq \frac{|z-w|}{2}$, or $|x-w| \geq \frac{|z-w|}{2}$.
If $|x-z| \geq \frac{|z-w|}{2}$, then
\[G_{t-r}(x-z) \leq G_{t-r}\left(\frac{z-w}{2}\right) = 2G_{4(t-r)}(z-w).\]
Similarly, if $|x-w|\geq  \frac{|z-w|}{2}$, then
$G_{t-\theta}(x-w) \leq 2G_{4(t-\theta)}(z-w)$.

These observations imply that the product is bounded as follows:
\[G_{t-r}(x-z)G_{t-\theta}(x-w) \leq 2G_{t-r}(x-z)G_{4(t-\theta)}(z-w) + 2G_{4(t-r)}(z-w)G_{t-\theta}(x-w).\]

{Since $|z-w|>1$, we have
$|z-w|^2>\frac{1}{2}+\frac{|z-w|^2}{2}$, and hence
}
\[G_{4(t-\theta)}(z-w) = (8\pi)^{-\frac{1}{2}}(t-\theta)^{-\frac{1}{2}} e^{-\frac{|z-w|^2}{8(t-\theta)}}
\leq (8\pi)^{-\frac{1}{2}}(t-\theta)^{-\frac{1}{2}} e^{-\frac{1}{16(t-\theta)}} e^{-\frac{|z-w|^2}{16(t-\theta)}}.
\]

We can remove the time singularities using the fact that $\sup_{x>0}x^{\frac{1}{2}}e^{-x}=\frac{1}{\sqrt{2 e}}$:
\[
(8\pi)^{-\frac{1}{2}}(t-\theta)^{-\frac{1}{2}} e^{-\frac{1}{16(t-\theta)} } 
\leq \sqrt{\frac{2}{\pi}} \cdot \frac{1}{\sqrt{2e}} \leq \pi^{-\frac{1}{2}}.
\]
Therefore, for any $\theta \in [0,t]$ and $|z-w|>1$,
\[
G_{4(t-\theta)}(z-w) \leq  \pi^{-\frac{1}{2}} e^{-\frac{|z-w|^2}{16T}} = 4 T^{\frac{1}{2}} G_{8T}(z-w).
\]
Similarly, for any $r \in [0,t]$ and $|z-w|>1$, we have:
$G_{4(t-r)}(z-w) \leq 4 T^{\frac{1}{2}}G_{8T}(z-w)$.
Therefore, when $|z-w|>1$, we have the bound:
\begin{equation}
	\label{eq:G-prod-w-z-big}
	G_{t-r}(x-z)G_{t-\theta}(x-w)
	\leq 8 T^{\frac{1}{2}}G_{8T}(z-w) \Big(G_{t-r}(x-z) + G_{t-\theta}(x-w)\Big).
\end{equation}
Finally, we can add \eqref{eq:G-prod-w-z-small} and \eqref{eq:G-prod-w-z-big} to obtain the desired conclusion \eqref{eq:kernel-product}.
\end{proof}

For the proof of the estimate on the second Malliavin derivative of the solution, we will use the following result, which is similar to Lemma \ref{lem:nualart123}.

\begin{lemma}
\label{lem2-nualart}
Let $(F_n)_{n\geq 1}$ be a sequence in $\bD^{2,2}$ such that $F_n \to F$ in $L^2(\Omega)$, and
\[
\sup_{n\geq 1} \bE \|D^2F_n\|_{\cH^{\otimes 2}}^2<\infty.
\]
 Then $F \in \bD^{2,2}$ and $(D^2F_n)_{n\geq 1}$ converges to $D^2F$ in the weak topology of $L^2(\Omega;\cH^{\otimes 2})$.
  \end{lemma}

\begin{proof} 
{
To prove that $F \in \bD^{2,2}$, we will use criterion \eqref{crit}. By \eqref{Dk-Jn}, for any $k \geq 2$,
\[
\bE\|D^2(J_k F_n)-D^2(J_k F)\|_{\cH^{\otimes 2}}^2 =k(k-1)\bE|J_k(F_n-F)|^2 \leq k(k-1)\bE|F_n-F|^2 \to 0,
\]
as $n \to \infty$. By Fatou's lemma,
\begin{align*}
\sum_{k\geq 2} \bE \|D^2 (J_k F)\|_{\cH^{\otimes 2}}^2&=\sum_{k\geq 2} \lim_{n \to \infty}
\bE \|D^2 (J_k F_n)\|_{\cH^{\otimes 2}}^2 \leq \liminf_{n \to \infty}\sum_{k\geq 2}
\bE \|D^2 (J_k F_n)\|_{\cH^{\otimes 2}}^2\\
&=\liminf_{n\to \infty}\bE \|D^2 F_n\|_{\cH^{\otimes 2}}^2 \leq \sup_{n\geq 1}\bE \|D^2 F_n\|_{\cH^{\otimes 2}}^2<\infty.
\end{align*}
Hence, $F \in \bD^{2,2}$. 

To prove the second statement, we proceed as in the proof of Proposition 3.5 of \cite{sanz05}. A corollary of Banach-Alaoglu theorem states that in a Hilbert space, any bounded sequence has a weakly convergent subsequence. In our case, the sequence $(D^2 F_n)_{n\geq 1}$ is bounded in $L^2(\Omega;\cH^{\otimes 2})$, and hence, it converges weakly to some $\eta \in L^2(\Omega;\cH^{\otimes 2})$,  along a subsequence $N' \subset \bN$. This means that for any $G \in  L^2(\Omega;\cH^{\otimes 2})$,
\[
\bE \langle D^2 F_n,G \rangle_{\cH^{\otimes 2}} \to \bE \langle \eta,G \rangle_{\cH^{\otimes 2}} \quad \mbox{as $n \to \infty,n \in N'$}.
\]
In particular, this holds for the process $J_k G=\{J_k G(\xi_1,\xi_2);\xi_1,\xi_2 \in \bR_{+}\times \bR \}$, which is also in $L^2(\Omega;\cH^{\otimes 2})$. Hence, for any $G \in  L^2(\Omega;\cH^{\otimes 2})$ and $k \geq 2$,
\[
\bE \langle D^2 F_n,J_k G \rangle_{\cH^{\otimes 2}} \to \bE \langle \eta,J_k G \rangle_{\cH^{\otimes 2}} \quad \mbox{as $n \to \infty,n \in N'$}.
\]

On the other hand, if $\delta_2: {\rm Dom}(\delta_2) \subseteq L^2(\Omega;\cH^{\otimes 2}) \to L^2(\Omega)$ is the adjoint of $D^2$, then by duality and the fact that $F_n \to F$ in $L^2(\Omega)$, we have:
\[
\bE \langle D^2 F_n,J_k G \rangle_{\cH^{\otimes 2}}=\bE\big[F_n \delta_2(J_k G)\big] 
\to \bE\big[F \delta_2(J_k G)\big] =\bE \langle D^2 F, J_k G \rangle_{\cH^{\otimes 2}}, \quad \mbox{as $n \to \infty$}.
\]
Here we used the fact that $J_k G \in {\rm Dom}(\delta_2)$. 
It follows that for any $G \in  L^2(\Omega;\cH^{\otimes 2})$, $\bE \langle \eta,J_k G \rangle_{\cH^{\otimes 2}}=\bE \langle D^2 F, J_k G \rangle_{\cH^{\otimes 2}}$ for any $k \geq 2$, and therefore,
\[
\bE \langle \eta,G \rangle_{\cH^{\otimes 2}}=\bE \langle D^2 F,  G \rangle_{\cH^{\otimes 2}}.
\]
This implies that $\eta=D^2 F$. We proved that every subsequence of $(D^2F_n)_{n\geq 1}$ which converges weakly in $L^2(\Omega;\cH^{\otimes 2})$, in fact converges weakly to $D^2F$. Therefore, the entire sequence converges weakly to $D^2F$.
}
\end{proof}

\begin{theorem}
\label{key-D2-th}
If Assumption \ref{assumpt2} holds, then
for any $p\geq 2$ and $T>0$, there exists a constant $\cC^{(2)}_{T,p}>0$ depending on $(p,T, L_b,L_\sigma,\|\sigma''\|_\infty, \|b''\|_\infty)$, such that for any $0<\theta<r<t<T$ and $x,z,w\in \bR$,
  	\[
  	\|D^2_{(r,z),(\theta,w)}u(t,x)\|_p \leq  \cC_{T,p}^{(2)} (1+|r-\theta|^{-\frac{1}{2}})G_{8T}(z-w) \big(G_{t-r}(x-z)+G_{t-\theta}(x-w)\big).
  	\]
  \end{theorem}

\begin{proof} 
Taking the Malliavin derivative to both sides of \eqref{rec-Dun}, and using the Heisenberg commutation principle \eqref{Heisenberg} and the chain rule \eqref{chain}, we obtain that:
\begin{align} 
D_{(r,z),(\theta,w)}^2 u_{n+1}(t,x) &=G_{t-r}(x-z)\s'\big(u_n(r,z)\big) D_{\theta,w} u_n(r,z)+\nonumber\\
& \int_r^t \int_{\bR} G_{t-s}(x-y) b''\big(u_n(s,y)\big) D_{\theta,w}u_n(s,y) D_{r,z}u_n(s,y) dyds+\nonumber\\
& \int_r^t \int_{\bR} G_{t-s}(x-y) b'\big(u_n(s,y)\big) D_{(r,z),(\theta,w)}^2 u_n(s,y) dyds+\nonumber\\
& \int_r^t \int_{\bR} G_{t-s}(x-y) \s''\big(u_n(s,y)\big) D_{\theta,w}u_n(s,y) D_{r,z}u_n(s,y)W(ds,dy)+\nonumber\\
\label{rec-D2un}
& \int_r^t \int_{\bR} G_{t-s}(x-y) \s'\big(u_n(s,y)\big) D_{(r,z),(\theta,w)}^2 u_n(s,y) W(ds,dy).
\end{align}
  
{
Using Minkowski inequality, followed by the inequality $(\sum_{i=1}^ 5 a_i)^2 \leq 5 \sum_{i=1}^5 a_i^2$, we get:}
  	\begin{equation} 
  		\label{eq:second-deriv-reprise}
  		\|D^2_{(r,z),(\theta,w)}u_{n+1}(t,x)\|_p^2 \leq 5\sum_{i=1}^5 \mathcal{A}_i,
  	\end{equation}
  	where 
  	\begin{align*}
  		\cA_1 &=  G_{t-r}^2(x-z) \|\s'\big(u_n(r,z)\big) D_{\theta,w} u_n(r,z) \|_p^2 \\
  		\cA_2 &= \left\|\int_r^t \int_{\bR} G_{t-s}(x-y) b''\big(u_n(s,y)\big) D_{\theta,w}u_n(s,y) D_{r,z}u_n(s,y) dyds\right\|_p^2 \\
  		\cA_3 &= \left\|\int_r^t \int_{\bR} G_{t-s}(x-y) b'\big(u_n(s,y)\big) D_{(r,z),(\theta,w)}^2 u_n(s,y) dyds\right\|_p^2 \\
  		\cA_4 &=  \left\|\int_r^t \int_{\bR} G_{t-s}(x-y) \s''\big(u_n(s,y)\big) D_{\theta,w}u_n(s,y) D_{r,z}u_n(s,y)W(ds,dy)\right\|_p^2\\
  		\cA_5 & = \left\|\int_r^t \int_{\bR} G_{t-s}(x-y) \s'\big(u_n(s,y)\big) D_{(r,z),(\theta,w)}^2 u_n(s,y) W(ds,dy)\right\|_p^2.
  	\end{align*}
 
{ 	
We recall that  $(L_b, L_{\s})$ and $\cC_{T,p}^{(1)}$ are given by \eqref{def-Lbs} and \eqref{key}.
}

\medskip
  	
$\bullet$ \underline{We treat $\cA_{1}$.} Using the fact that $\sigma'$ is bounded and the key estimate \eqref{key-un} for $Du_n$,
  	\[
  	\cA_{1} \leq  {\big(\cC_{T,p}^{(1)}\big)^2}  L_{\sigma}^2  G^2_{t-r}(x-z)G^2_{r-\theta}(z-w).
  	\]
{
  	Since $r-\theta<r<t<8t<8T$, by \eqref{eq:gaussian-bound}, 
 $G_{r-\theta}(z-w) \leq (8T)^{1/2} (r-\theta)^{-1/2}  G_{8T}(z-w)$.}
Hence,
  	\begin{equation} 
  		\label{eq:A1-bound}
  		\cA_1 \leq 8 T \big(\cC_{T,p}^{(1)}\big)^2  L_{\sigma}^2  \, (r-\theta)^{-1} G^2_{8T}(z-w)  G^2_{t-r}(x-z) .
  	\end{equation}
  
$\bullet$ \underline{We treat $\cA_{2}$.} 
By Minkowski's inequality, 
  	\begin{align*}
  		&\left\|\int_{r}^t \int_{\mathbb{R}}G_{t-s}(x-y) b''(u_n(s,y))D_{r,z}u_n(s,y)D_{\theta,w}u_n(s,y)dyds \right\|_p^2\nonumber\\
  		&\leq  \left(\int_{r }^t \int_{\mathbb{R}}G_{t-s}(x-y) \|b''(u(s,y))D_{r,z}u_n(s,y)D_{\theta,w}u_n(s,y)\|_pdyds\right)^2.
  	\end{align*}
Using Cauchy-Schwarz inequality, the fact that $b''$ is bounded, and the key estimate \eqref{key-un},  
  	\begin{align}
  \nonumber
  		\|b''(u(s,y))D_{r,z}u_n(s,y)D_{\theta,w}u_n(s,y)\|_p & \leq \|b''\|_\infty \|D_{r,z}u_n(s,y)\|_{2p}\|D_{\theta,w}u_n(s,y)\|_{2p} \nonumber\\
  \label{b-sec}
  		&\leq   \|b''\|_\infty  {\big(\cC_{T,2p}^{(1)}\big)^2}   G_{s-r}(y-z)G_{s-\theta}(y-w).
  	\end{align}
  	Therefore, 
  	\begin{align*}
  		&\cA_2 \leq  \|b''\|_{\infty}^2 \big(\cC_{T,2p}^{(1)}\big)^4  \left(\int_r^t \int_{\mathbb{R}}G_{t-s}(x-y)G_{s-r}(y-z)G_{s-\theta}(y-w)dyds\right)^2.
  	\end{align*}
  	
  	We use Lemma \ref{lem:kernel-product} to bound the product $G_{s-r}(y-z)G_{s-\theta}(y-w)$: 
\begin{align}
\nonumber  
& G_{s-r}(y-z)G_{s-\theta}(y-w) \leq \\
\label{prod}
& \quad \quad \quad 8T^{1/2}\big(1 + e^{\frac{1}{8T}}\big) \big(1 + (r-\theta)^{-\frac{1}{2}}\big) G_{8T}(z-w)\Big(G_{s-r}(y-z) + G_{s-\theta}(y-w)\Big).
\end{align}

Hence,  
  	\begin{align*}
  		\cA_2 \leq &  64\, T \, \|b''\|_{\infty}^2 \big(\cC_{T,2p}^{(1)}\big)^4   \big(1 + e^{\frac{1}{16T}}\big)^2(1 + \big(r-\theta)^{-\frac{1}{2}}\big)^2 \, G^2_{8T}(z-w) \nonumber\\
  		&\times\Bigg(\int_r^t \int_{\mathbb{{R}}}G_{t-s}(x-y)
  		\Big(G_{s-r}(y-z) 
  		+ G_{s-\theta}(y-w)\Big)dyds\Bigg)^2.
  	\end{align*}
  	
  	We can now use the semigroup property for the $dy$ integral, and so,
  	\begin{align} 
  \nonumber
  		\cA_2 \leq & {128} \, T^3 \, \|b''\|_{\infty}^2 \big(\cC_{T,2p}^{(1)}\big)^4 \big(1 + e^{\frac{1}{16T}}\big)^2 {\big(1 + (r-\theta)^{-1}\big)} G^2_{8T}(z-w)  \\
  \label{eq:A2-bound}
 &  \quad \quad \big(G_{t-r}(x-z) 
  		+G_{t-\theta}(x-w)\big)^2.
  	\end{align}

 $\bullet$ \underline{We treat $\cA_3$.} Using Minkowski's inequality, followed by H\"older's inequality for the measure $G_{t-s}(x-y)dyds$, and the assumption that $b'$ is bounded, we have:
  	\begin{equation} 
  		\label{eq:A3-bound}
  		\cA_3  	\leq L_b^2\left(\int_r^t \int_{\bR} G_{t-s}(x-y)\|D^2_{(r,z),(\theta,w)}u_n(s,y)\|_pdyds\right)^2.
  	\end{equation}

$\bullet$ \underline{We treat $\cA_4$.} Using BDG inequality \eqref{BDG}, followed by Minkowski's inequality,
  	\begin{align*}
  		\cA_4 & \leq z_{p}^2 \left\| \int_{r }^t \int_{\bR} G^2_{t-s}(x-y)  | \sigma''(u_n(s,y)) D_{r,z}u_n(s,y)D_{\theta,w}u_n(s,y)|^2dyds \right\|_{p/2} \\
  		& \leq
  		z_{p}^2  \int_{r }^t \int_{\mathbb{R}} G^2_{t-s}(x-y)  \| \sigma''(u_n(s,y)) D_{r,z}u_n(s,y)D_{\theta,w}u_n(s,y)\|_p^2dyds.
  	\end{align*}
{Similarly to \eqref{b-sec}, by Cauchy-Schwarz inequality and \eqref{key-un},
  	\begin{align*}
  		\|\s''(u(s,y))D_{r,z}u_n(s,y)D_{\theta,w}u_n(s,y)\|_p  \leq  \|\s''\|_\infty  \big(\cC_{T,2p}^{(1)}\big)^2   G_{s-r}(y-z)G_{s-\theta}(y-w).
  	\end{align*}
  }
  	Therefore,   
  	\begin{equation}
  		\label{A4-int}
  		\cA_4 \leq z_p^2 \|\s''\|_{\infty}^2 { \big(\cC_{T,2p}^{(1)}\big)^4 } \int_{r }^t \int_{\mathbb{R}} G^2_{t-s}(x-y)G^2_{s-r}(y-z)G^2_{s-\theta}(y-w)dyds.
  	\end{equation}
  	
{	
Now we use \eqref{prod} to bound the product $G_{s-r}^2(y-z)G_{s-\theta}^2(y-w)$.} Hence,
\begin{align}
\nonumber
\cA_4 & \leq 64  \, z_p^2 \, \|\s''\|_{\infty}^2  \big(\cC_{T,2p}^{(1)}\big)^4  \, T  \big(1 + e^{\frac{1}{8T}}\big)^2 \big(1 + (r-\theta)^{-\frac{1}{2}}\big)^2 G_{8T}^2 (z-w)\\
\label{A4-1}
& \quad
\int_{r }^t \int_{\bR} G^2_{t-s}(x-y) \Big(G_{s-r}(y-z) + G_{s-\theta}(y-w)\Big)^2 dyds.
\end{align}
 
{To bound the integral above,} we
define:
\begin{equation}
\label{def-rho}
\rho(t,x) =G_{t-r}(x-z) + G_{t-\theta}(x-w).
\end{equation}
Note that $\rho$ is a solution of the heat equation, and can be written as 
  	\begin{equation}
  		\label{eq:rho-A4}
  		\rho(t,x) = \int_{\mathbb{R}} G_{t-r}(x-y)\mu(dy), \quad \text{ with } \quad \mu(dy) = \delta_z(dy) + G_{r-\theta}(y-w)dy.
  	\end{equation}
{Using the fact that $\sup_{\|h\|_{\cH}=1}\langle f,h \rangle_{\cH}=\|f\|_{\cH}$ for any $f \in \cH$, we have:}
  	\begin{align*}
&  		\int_r^t \int_{\bR}G_{t-s}^2(x-y) \rho^2(s,y)dyds 
  		 =\sup_{\|h\|_{\cH}=1} \left(\int_r^t \int_{\bR} G_{t-s}(x-y)\rho(s,y)h(s,y)dyds\right)^2 \\
  & \quad \quad \quad = \sup_{\|h\|_{\cH}=1} \left(\int_{\bR}\int_r^t \int_{\bR}  G_{t-s}(x-y){G_{s-r}(y-z')}h(s,y)dyds {\mu(dz')}\right)^2 \\
  &\quad \quad \quad \leq \left[\int_{\bR} \left(\int_r^t \int_{\bR} G^2_{t-s}(x-y)G^2_{s-r}(y-z'){dy} ds \right)^{\frac{1}{2}} \mu(dz')\right]^2\\
  & \quad \quad \quad {=} \left(\frac{\pi}{4}\right)^{\frac{1}{2}} (t-r)^{\frac{1}{2}} \left(\int_{\bR}  G_{t-r}(x-z')\mu(dz')\right)^2 =
  \left(\frac{\pi}{4}\right)^{\frac{1}{2}} (t-r)^{\frac{1}{2}} \rho^2(t,x),
  	\end{align*} 
using the Cauchy-Schwarz inequality, relation \eqref{G2-G2-2}, and the definition of $\rho$. 
{Returning to \eqref{A4-1}, we obtain:}
  	\begin{align}
  		\nonumber
  		\cA_4 & \leq {64} \sqrt{\pi} \, T^{\frac{3}{2}} \, z_p^2 \|\s''\|_{\infty}^2  \big(\cC_{T,2p}^{(1)}\big)^4   \big(1 + e^{\frac{1}{8T}}\big)^2 {\big(1 + (r-\theta)^{-1}\big)} G_{8T}^2 (z-w)\\
  \label{eq:A4-bound}
  & \quad \quad 
   \Big(G_{t-r}(x-z) + G_{t-\theta}(x-w) \Big)^2.
  	\end{align}

$\bullet$ \underline{We treat $\cA_5$.}	
Using the BDG inequality \eqref{BDG}, followed by Minkowski's inequality and the fact that {$\sigma'$} is bounded, we have:
  	\begin{align}
  		\nonumber
  		\cA_5 &\leq z_p^2 \left\|\int_r^t \int_{\bR} G_{t-s}^2(x-y)|\s'(u_n(s,y)) D_{(r,z),(\theta,w)}^2 u_n(s,y)|^2 dy ds\right\|_{p/2} 
  		\\
  		\label{eq:A5-bound}
  		&\leq  z_{p}^2 {L_{\s}^2} \int_{r}^t \int_{\bR} G^2_{t-s}(x-y)\|D^2_{(r,z),(\theta,w)}u_n(s,y)\|_p^2dyds. 	
  	\end{align}

\medskip
  	
  	We now return to \eqref{eq:second-deriv-reprise}. Using the fact that $C_{t,p}$ is non-decreasing in $t$, we combine \eqref{eq:A1-bound}, \eqref{eq:A2-bound}, \eqref{eq:A3-bound}, \eqref{eq:A4-bound} and \eqref{eq:A5-bound}. {We group together $\cA_1$, $\cA_2$ and $\cA_4$.} We conclude that for any $0<\theta<r< T$ and $z,w\in \bR$ fixed,
  	\begin{align*}
  		& \|D^2_{(r,z),(\theta,w)}u_{n+1}(t,x)\|_p^2 \\
  		&\leq  5\left\{ A_T^2 \rho^2(t,x)+ {L_b^2} \left(\int_r^{t} \int_{\mathbb{R}} G_{t-s}(x-y)\|D^2_{(r,z),(\theta,w)}u_n(s,y)\|_pdyds\right)^2 \right. \\
  		&\qquad \left. + z_p^2 {L_{\s}^2} \int_{r}^{t} \int_{\mathbb{R}} G^2_{t-s}(x-y)\|D^2_{(r,z),(\theta,w)}u_n(s,y)\|_p^2dyds\right\}, \quad \mbox{for any} \quad t \in [r,T],x\in \bR
  	\end{align*}
  	where $\rho$ is given by \eqref{def-rho},
  	\begin{align*}
  		A_{T}^2 &:={B_{T}^2 } \big(1+ (r-\theta)^{-1}\big)G^2_{8T}(z-w),
  	\end{align*}
  and

\begin{equation}
\label{def-BT}
B_{T}^2=8 T \big(\cC_{T,p}^{(1)}\big)^2  L_{\sigma}^2 +
64\left(2 \, T^3 \, \|b''\|_{\infty}^2 +
    \sqrt{\pi} \, T^{\frac{3}{2}} z_p^2 \|\s''\|_{\infty}^2 \right) \big(\cC_{T,2p}^{(1)}\big)^4   \big(1 + e^{\frac{1}{8T}}\big)^2.
\end{equation}

The function $f_n(t,x) = \|D_{(r,z),(\theta,w)}^2 u_n(t,x)\|_p$ satisfies the assumptions of Lemma \ref{lemA2-HNV} {on $I=(r,T]$}, with $\rho(t,x)$ given by \eqref{def-rho}, $A=A_{T}$ and $C_t=5C_{p,L_b,L_{\sigma}}$, where $C_{p,L_b,L_{\sigma}}$ is given by \eqref{def-CpL}. Therefore, {there exists a constant $\cC_{T,p}^{(3)}>0$ depending on $(p,T,L_{b},L_{\sigma},\|b''\|_{\infty},\|\sigma''\|_{\infty})$,} such that for any $n\geq 1$, $0<r < t\leq T$ and $x,z \in \bR$,
  	\begin{equation}
  \label{D2-un}
  	\|D^2_{(r,z),(\theta,w)}u_n(t,x)\|_p \leq \cC_{T,p}^{(3)}\, A_{T} \rho(t,x). 
  	\end{equation}
{We let $\cC_{T,p}^{(2)}=\cC_{T,p}^{(3)}B_T$.}
  
\medskip

We apply Lemma \ref{lem2-nualart} to $F_n=u_n(t,x)$ and $F=u(t,x)$. 
It follows that $u(t,x)\in \bD^{2,2}$ and $\{D^2 u_n(t,x)\}_{n\geq 1}$ converges to $D^2 u(t,x)$ in the weak topology of $L^2(\Omega;\cH^{\otimes 2})$. Then, the conclusion follows by applying Lemma \ref{lem:weak-implies-pointwise} {with $H=\cH^{\otimes 2}=L^2((\bR_{+}\times \bR)^2)$} and relation \eqref{D2-un}.
In our application of this lemma, 
\[
\phi((r,z),(\theta,w))=\cC_{T,p}^{(2)} \big(1+ (r-\theta)^{-1}\big)^{1/2} G_{8T}(z-w)\Big(G_{t-r}(x-z)+G_{t-\theta}(x-w)\Big).
\]

\end{proof}
  
Next, we prove a uniform bound on spatial integrals of the second derivative.
  
\begin{theorem} 
\label{key2-th}
For any $p\geq 2$,  $0<\theta,r<t \leq T$ and $w,z \in \bR$,
\begin{equation}
\sup_{R>0}\left\|\int_{-R}^RD^2_{(r,z),(\theta,w)}u(t,x)dx\right\|_p \leq 2 {\cC_{T,p}^{(2)}}  \big( 1+  |r-\theta|^{-\frac{1}{2}}\big) G_{8T}(z-w),
\end{equation}
where $\cC_{T,p}^{(2)}$ is the constant from Theorem \ref{key-D2-th}.
\end{theorem}
  
\begin{proof}
Assume without loss of generality that $\theta<r$. By Minkowski's inequality,
\begin{align*}
\sup_{R>0}\left\|\int_{-R}^RD^2_{(r,z),(\theta,w)}u(t,x)dx\right\|_p
\leq \int_{\bR} \left\|D^2_{(r,z),(\theta,w)}u(t,x)\right\|_p dx.
\end{align*}
By Theorem \ref{key-D2-th},
\begin{align*}
  		&\int_{\bR} \left\|D^2_{(r,z),(\theta,w)}u(t,x)\right\|_p dx \\
  		&\leq \cC_{T,p}^{(2)} \big(1 + |r-\theta|^{-\frac{1}{2}}\big)G_{8T}(z-w) \int_{\bR} \big(G_{t-r}(x-z) + G_{t-\theta}(x-w)\big)dx\\
  		&=2 \cC_{T,p}^{(2)} \big( 1 + |r-\theta|^{-\frac{1}{2}} \big)G_{8T}(z-w).
  	\end{align*}

  	\end{proof}

\section{Ergodicity and limiting covariance}
\label{section-first-order}

In this section, we give the proof of Theorem \ref{main1}.

For ergodicity, we recall the following result, {which a variant of Lemma 7.2 of \cite{CKNP21}.}

\begin{lemma}[Lemma 4.2 of \cite{BZ24}]
	\label{erg-lem}
	A strictly stationary process $\{Y(x)\}_{x \in \bR^d}$ is ergodic if 
	\[
	\lim_{R \to \infty}\frac{1}{R^{2d}}{\rm Var}\left( \int_{[0,R]^d} g\Big( \sum_{j=1}^{k}b_j Y(x+\zeta_j) \Big)dx\right)=0,
	\]
	for all integers $k\geq 1$, for every $b_1,\ldots,b_k,\zeta_1,\ldots,\zeta_k \in \bR^d$ and $g(x)=\cos x$ or $g(x)=\sin x$.
\end{lemma}

Based on this result, we have the following lemma.

\begin{lemma}
\label{Eulalia}
Let $\{Z(t,x);t\geq 0,x\in \bR\}$ be an adapted random field such that $\{Z(t,x)\}_{x\in \bR}$ is strictly stationary, and 
$Z(t,x) \in \bD^{1,2}$ for any $t\geq 0$ and $x \in \bR$. Assume that for any $0<r<t$ and $x,z \in \bR$, 
\[
\|D_{r,z}Z(t,x)\|_{2} \leq C_{t} G_{t-r}(x-z)
\]
where $C_{t}>0$ is a constant that depends on $t$. Then $\{Z(t,x)\}_{x\in \bR}$ is ergodic.
\end{lemma}

\begin{proof}

	We apply Lemma \ref{erg-lem}. Let $b_1,\ldots,b_k \in \bR$ and $\zeta_1,\ldots, \zeta_k\in \bR$ be arbitrary. Assume that $g(x)=\cos x$ or $g(x)=\sin x$. We apply Poincar\'e inequality \eqref{Poincare}, {identity \eqref{Mal-zero}, the fact that $D_{r,z}$ commutes with the $dx$ integral}, and Minkowski's inequality. We deduce that
	\begin{align*}
		V_R:=& {\rm Var}\left( \int_0^R g \Big(\sum_{j=1}^{k} b_j Z(t,x+\zeta_j)\Big) dx\right)\leq {\bE\left\|D\left(\int_0^R g \Big(\sum_{j=1}^{k} b_j Z(t,x+\zeta_j)\Big) dx \right) \right\|_{\cH}^2 } \\
&=\bE \int_0^t \int_{\bR}\left|\int_0^R D_{r,z} g\Big(\sum_{j=1}^{k} b_j Z(t,x+\zeta_j)\Big) dx \right|^2 dzdr\\
		& =  \int_0^t \int_{\bR}\left\|\int_0^R D_{r,z} g\Big(\sum_{j=1}^{k} b_j Z(t,x+\zeta_j)\Big) dx \right\|_2^2 dzdr\\
		& \leq   \int_0^t \int_{\bR}\left(\int_0^R \left\|D_{r,z} g\Big(\sum_{j=1}^{k} b_j Z(t,x+\zeta_j)\Big)\right\|_2 dx \right)^2 dzdr.
	\end{align*}
	By the chain rule \eqref{chain} for the Malliavin derivative,
	\[
	D_{r,z} g\big(\sum_{j=1}^{k} b_j Z(t,x+\zeta_j)\big)=g'\Big(\sum_{j=1}^{k} b_j Z(t,x+\zeta_j) \Big) D_{r,z}\Big(\sum_{j=1}^{k} b_j Z(t,x+\zeta_j)\Big).
	\]
	Using the fact that $|g'|\leq 1$, followed by triangular inequality and the key estimate \eqref{key}, we obtain:
	\begin{align*}
		\big\|D_{r,z} g\big(\sum_{j=1}^{k} b_j Z(t,x+\zeta_j)\big)\big\|_2 & \leq \big\|D_{r,z}\big(\sum_{j=1}^{k} b_j Z(t,x+\zeta_j)\big)\big\|_2  \leq \sum_{j=1}^{k}|b_j| \big\|D_{r,z}Z(t,x+\zeta_j)\big\|_2 \\
		& \leq C_{t,2} \sum_{j=1}^{k}|b_j| G_{t-r}(x+\zeta_j-z).
	\end{align*}
	Hence
	\begin{align*}
		V_R & \leq C_{t} \int_0^t \int_{\bR} \left( \sum_{j=1}^{k}|b_j|  \int_0^R  G_{t-r}(x+\zeta_j-z) dx \right)^2 dzdr\\
		& \leq C_{t} k \sum_{j=1}^{k} b_j^2 \int_0^t \int_{\bR} \left(  \int_0^R  G_{t-r}(x+\zeta_j-z) dx \right)^2 dzdr\\
		& \leq C_{t} k \sum_{j=1}^{k} b_j^2 \int_0^t \int_{\bR}  \int_0^R  G_{t-r}(x+\zeta_j-z) dxdzdr= C_{t} k Rt\sum_{j=1}^{k} b_j^2.
	\end{align*}
	From this estimate, we obtain the desired conclusion: $V_R/R^2 \to 0$ as $R \to \infty$.
\end{proof}


\bigskip

{\bf Proof of Theorem \ref{main1}:} Part (i) follows from Lemma \ref{Eulalia} and the key estimate \eqref{key} for the Malliavin derivative $Du$.

{For part (ii), we proceed as on page 27 of \cite{NZ20}.}
Since $\rho_{t,s}(x-y)={\rm Cov}(u(t,x),u(t,y))$,
\begin{align*}
	\bE[F_R(t) F_R(s)]&=\int_{-R}^R \int_{-R}^R \rho_{t,s}(x-y)dxdy= \int_{\bR} \rho_{t,s}(z) \left(\int_{\bR} 1_{[-R,R]}(x) 1_{[-R,R]}(x-z) dx \right) dz\\
	& = \int_{\bR}\rho_{t,s}(z) {\rm Leb}\big([-R,R] \cap [-R+z,R+z]\big)dz=\int_{\{|z|\leq 2R \}}\rho_{t,s}(z) (2R-|z|)dz,
\end{align*}
where ${\rm Leb}$ is the Lebesgue measure. Then, by the dominated convergence theorem,
\[
\frac{1}{R}\bE[F_R(t) F_R(s)]=\int_{\{|z|\leq 2R \}}\rho_{t,s}(z) \left(2-\frac{|z|}{R} \right)dz \to 2\int_{\bR}\rho_{t,s}(z)dz, \quad \mbox{as $R \to \infty$},
\]
provided that 
\begin{equation}
	\label{rho-int}
	\int_{\bR}|\rho_{t,s}(x)|dx<\infty.
\end{equation}

To prove \eqref{rho-int}, we will use \eqref{Poin-cov}, \eqref{key}, identity \eqref{Mal-zero}, {and the semigroup property.} So,
\begin{align*}
	|\rho_{t,s}(x)| & \leq \int_{0}^{t \wedge s} \int_{\bR} \|D_{r,z} u(t,x)\|_{2} \|D_{r,z}u(s,0)\|_{2} dzdr\\
	& \leq  C \int_0^{t \wedge s} \int_{\bR}G_{t-r}(x-z)G_{s-r}(z)dzdr
	=C \int_{0}^{t\wedge s} G_{t+s-2r}(x)dr,
\end{align*}
where $C=\cC_{t,2}^{(1)} \cC_{s,2}^{(1)}$. Hence
\[
\int_{\bR}|\rho_{t,s}(x)|dx \leq C\int_{0}^{t\wedge s} \int_{\bR}G_{t+s-2r}(x)dxdr=C(t\wedge s).
\]

\section{Quantitative CLT}
\label{section-QCLT}

In this section, we give the proof of Proof of Theorem \ref{main-th}.(i). For this, we will use the second-order Poincar\'e inequality due to Vidotto \cite{vidotto20}.

 We recall that the {\em first contraction} between functions {$\varphi,\psi \in \cH \otimes \cH $} is given by:
	\begin{equation} \label{eq:contraction-def}
		(\varphi \otimes_1 \psi)\big((t,x),(s,y)\big) := \int_0^\infty \int_{\mathbb{R}} \varphi(\theta,w,t,x)\psi(\theta,w,s,y)dwd\theta.
	\end{equation}

\begin{proposition}[Proposition 3.2 of \cite{vidotto20}]
\label{prop18-1}
Let $F \in \bD^{2,4}$ be such that $\bE(F)=0$ and $\bE(F^2)=\sigma^2$. Then
\[
{\rm Var}\big(\langle DF,-DL^{-1}F\rangle_{\cH}\big) =\bE\left[\big|\sigma^2 -\langle DF,-DL^{-1}F\rangle_{\cH}\big|^2\right] \leq 4 \cA,
\]
where
\begin{equation}
\label{def-A}
\cA=\int_{(\bR_{+} \times \bR)^2} \|[D^2 F\otimes_1 D^2 F]((r,z),(s,y))\|_2 \, \|D_{r,z}FD_{s,y}F\|_2 drdz \, dsdy. 
\end{equation}

\end{proposition}

\begin{corollary}
\label{prop18}
Let $F \in \bD^{2,4}$ be such that $\bE(F)=0$ and $\bE(F^2)=\sigma^2$. Then
\[
	d_{TV}\left(\frac{F}{\sigma}, Z\right) =d_{TV}(F,\sigma Z)
\leq \frac{4}{\sigma^2}\sqrt{\cA},
	\]
	where $Z \sim N(0,1)$ and $\cA$ is given by \eqref{def-A}.
\end{corollary}

\begin{proof}

By \eqref{Mal-cor}, we write $F=\delta(v)$ with $v=-DL^{-1}F$. 
Then we use the Stein-Malliavin bound \eqref{SM}. To estimate the variance in \eqref{SM}, we use Proposition \ref{prop18-1}.

\end{proof}

\bigskip
{\bf Proof of Theorem \ref{main-th}.(i):}
By Corollary \ref{prop18} and Theorem \ref{main1}.(ii),
\[
d_{TV}\left( \frac{F_R(t)}{\sigma_{R}(t)},Z\right) \leq \frac{4}{\sigma_R^2(t)}\sqrt{\cA_R} \leq \frac{C_t^*}{R} \sqrt{\cA_R},
\]
where $C_t^*>0$ is a constant depending on $t$, and
\[
\cA_R:=\int_{([0,t] \times \bR)^2} \|[D^2F_R(t) \otimes_1 D^2F_{R}(t)]((r,z),(s,y))\|_2 \, \|D_{r,z}F_R(t)D_{s,y}F_R(t)\|_2 drdz ds dy .
\]

We will prove that
\begin{equation}
	\label{cAR}
	\cA_R \leq C_t R,
\end{equation}
where $C_t>0$ is a constant depending on $t$.

By Minkowski's inequality and the key estimate \eqref{key}, for any $p\geq 2$,
\[
	\|D_{r,z}F_R(t)\|_p \leq \int_{-R}^{R}\|D_{r,z}u(t,x)\|_p dx \leq \cC_{t,p}^{(1)} \int_{-R}^{R}G_{t-r}(x-z)dx.
\]
Therefore, using the Cauchy-Schwarz inequality, we infer that:
\begin{align} 
\nonumber
\|D_{r,z}F_R(t)D_{s,y}F_R(t)\|_2 & \leq \|D_{r,z}F_R(t)\|_4\|D_{s,y}F_R(t)\|_4 \\
\label{eq:DFR-L4}
& \leq
\big(\cC_{t,4}^{(1)} \big)^2 \int_{[-R,R]^2}G_{t-r}(x_1-z) G_{t-s}(x_2-y)dx_1dx_2.
\end{align}

By the definition of $F_R(t)$,
$D^2_{(r,z),(\theta,w)}F_R(t) = \int_{-R}^R D^2_{(r,z),(\theta,w)}u(t,x)dx$,
and hence,
\begin{align*}
	&[D^2F_R(t) \otimes_1 D^2F_{R}(t)]((r,z),(s,y)) \\
	&= \int_0^t \int_{\mathbb{R}}\int_{-R}^R D^2_{(r,z),(\theta,w)}u(t,x_3) dx_3 \int_{-R}^R  D^2_{(s,y),(\theta,w)}u(t,x_4) dx_4dwd\theta.
\end{align*}

By Minkowski's inequality, Cauchy-Schwarz inequality, Theorem \ref{key2-th} and the semigroup property,
\begin{align}
	\nonumber
	&\|[D^2F_R(t) \otimes_1 D^2F_{R}(t)]((r,z),(s,y)) \|_2 \\
	\nonumber
	&\leq \int_0^t \int_{\bR} \left\|\int_{-R}^R D^2_{(r,z),(\theta,w)}u(t,x_3)dx_3 \int_{-R}^R D^2_{(s,y),(\theta,w)}u(t,x_4)dx_4  \right\|_2dwd\theta \\
	\nonumber
	&\leq \int_0^t \int_{\bR} \left\|\int_{-R}^R D^2_{(r,z),(\theta,w)}u(t,x_1)dx_3\right\|_4 \left\| \int_{-R}^R D^2_{(s,y),(\theta,w)}u(t,x_2)dx_4  \right\|_4dwd\theta \\
	\nonumber
	&\leq {4\big(\cC_{t,4}^{(2)}\big)^2} \int_0^t \int_{\bR} \left(1 + |r-\theta|^{-\frac{1}{2}}\right)\left(1 +|s-\theta|^{-\frac{1}{2}}\right)G_{8t}(z-w)G_{8t}(y-w)dw\\
	&=4\big(\cC_{t,4}^{(2)}\big)^2 G_{16t}(z-y)\int_0^t \left(1 + |r-\theta|^{-\frac{1}{2}}\right)\left(1 +|s-\theta|^{-\frac{1}{2}}\right)d\theta \nonumber\\
	\label{eq:D2FR-contr-bound}
	&\leq 4\big(\cC_{t,4}^{(2)}\big)^2 (t+8t^{1/2})G_{16t}(z-y)\left(1+ |r-s|^{-\frac{1}{2}}\right).
\end{align}

For the last line above, we used the fact that for $r \in [0,t]$, 
{
\begin{equation}
\label{rs-int}
\int_0^t |r-\theta|^{-\frac{1}{2}} d\theta=\int_0^r (r-\theta)^{-\frac{1}{2}}d\theta+\int_r^t (\theta-r)^{-\frac{1}{2}}d\theta=2\big(r^{\frac{1}{2}}+(t-r)^{\frac{1}{2}}\big)\leq 4t^{\frac{1}{2}},
\end{equation}
}
and for $0<r<s<t$,
\begin{align*}
	&\int_0^t |r-\theta|^{-\frac{1}{2}} |s-\theta|^{-\frac{1}{2}}d\theta \\
	&=\int_0^r (r-\theta)^{-\frac{1}{2}}(s-\theta)^{-\frac{1}{2}}d\theta+\int_r^s (\theta-r)^{-\frac{1}{2}}(s-\theta)^{-\frac{1}{2}}d\theta+\int_s^t (\theta-r)^{-\frac{1}{2}}(\theta-s)^{-\frac{1}{2}}d\theta\\
	& \leq (s-r)^{-\frac{1}{2}} \int_0^{r} (r-\theta)^{-\frac{1}{2}}d\theta
	+ \int_{r}^{s} (s-\theta)^{-\frac{1}{2}}(\theta-r)^{-\frac{1}{2}}d\theta
	+(s-r)^{-\frac{1}{2}}\int_{s}^t (\theta-s)^{-\frac{1}{2}}d\theta \\
	&= 2(s-r)^{-\frac{1}{2}}r^{\frac{1}{2}}+\pi+2(s-r)^{-\frac{1}{2}}(t-s)^{\frac{1}{2}}  \leq (4+\pi)\,t^{\frac{1}{2}}(s-r)^{-\frac{1}{2}}\leq 8\,t^{\frac{1}{2}}(s-r)^{-\frac{1}{2}},
\end{align*}
using the Beta identity \eqref{beta} for the middle integral.

Now combining \eqref{eq:DFR-L4} and \eqref{eq:D2FR-contr-bound}, leads to the estimate
\[
\cA_R \leq C_t^{(1)} \int_{([0,t]\times \bR)^2} \int_{[-R,R]^2} \big( 1+ |r-s|^{-\frac{1}{2}}\big)G_{16t}(z-y)G_{t-r}(x_1-z)G_{t-s}(x_2-y) dx_1dx_2drdzdsdy,
\]
with {$C_{t}^{(1)}=4\big(\cC_{t,4}^{(1)}\big)^2 \big(\cC_{t,4}^{(2)}\big)^2 (t+8t^{1/2})$.}
By the semigroup property,
\[
\cA_R \leq C_t^{(1)} \int_{[0,t]^2} \int_{[-R,R]^2} \big( 1+ |r-s|^{-\frac{1}{2}}\big)G_{18t - r -s}(x_1-x_2)dx_1dx_2drds.
\]
Next, integrating in $x_1$ and $x_2$, {and using \eqref{rs-int}}, we obtain:
\[
	\cA_R \leq R C_t^{(1)} \int_{[0,t]^2} \big(1+|r-s|^{-\frac{1}{2}}\big)drds \leq RC_t^{(1)}
(t^2+4t^{\frac{3}{2}}).
\]
This concludes the proof of \eqref{cAR}.

\section{Functional CLT}
\label{section-FCLT}

In this section, we give the proof of Theorem \ref{main-th}.(ii). The proof is based on the classical method which requires to show finite-dimensional convergence and tightness. 

\medskip

{\bf Step 1.} (Finite dimensional convergence)
 We prove that for any $t_1,\ldots,t_m \in [0,T]$.
\[
\Big( \frac{1}{\sqrt{R}} F_R(t_1), \ldots, \frac{1}{\sqrt{R}} F_R(t_m)\Big) \stackrel{d}{\to}
\big( \cG(t_1), \ldots, \cG(t_m)\big), \quad \mbox{as $R \to \infty$},
\]
where $\stackrel{d}{\to}$ denotes convergence in distribution.
By Cram\'er-Wold theorem, we have to show that for any  $b_1,\ldots , b_m\in\bR$, 
\begin{equation}
\label{XR-conv}
X_R:=\frac{1}{\sqrt{R}}\sum_{j=1}^m b_j  F_R(t_j)
\stackrel{d}{\to}
\sum_{j=1}^m b_j  \cG_{t_j}\sim N(0,\tau^2), \quad \mbox{as $R \to \infty$},
\end{equation}
where $\tau^2={\rm Var}\big(\sum_{j=1}^m b_j  \cG_{t_j}\big)$,
that is
$X_R/\tau\stackrel{d}{\to} N(0,1)$ as $R \to \infty$.
By Theorem \ref{main1}.(ii), 
\[
\tau_R^2:={\rm Var}(X_R)=\frac{1}{R} \sum_{j,k=1}^{m}b_j b_k  \bE[F_R(t_j) F_R(t_k)]
\to \tau^2,  \quad \mbox{as $R \to \infty$}.
\]
Therefore, it is enough to prove that $X_R/\tau_R \stackrel{d}{\to} N(0,1)$ as $R\to \infty$.
We will prove something stronger, namely, if $Z \sim N(0,1)$, then
\begin{equation}
\label{fdd}
d_{TV}\left(\frac{X_R}{\tau_R},Z \right) \leq C R^{-1/2},
\end{equation}
where $C>0$ is a constant that depends on $(t_1,\ldots,t_m)$ and $(b_1,\ldots,b_m)$.

To prove \eqref{fdd}, we proceed as in the proof of the QCLT above. By Corollary \ref{prop18},
\begin{equation}
\label{Vido1}
d_{TV}\left(\frac{X_R}{\tau_R},Z \right) \leq \frac{4}{\tau_R^2}\sqrt{\cB_R},
\end{equation}
where 
\[
\cB_R=\int_{([0,T] \times \bR)^2} \|[D^2 X_R \otimes_1 D^2 X_R]((r,z),(s,y))\|_2 \, \|D_{r,z}X_R D_{s,y} X_R\|_2 drdz \, dsdy.
\]
we defined $T=\max\{t_1,\ldots,t_m\}$ and we used \eqref{Mal-zero}.
By Minkowski's inequality and \eqref{key},
\[
\|D_{r,z}X_R\|_p \leq \frac{1}{\sqrt{R}}\sum_{j=1}^{m}|b_j|\int_{-R}^R \|D_{r,z}u(t_j,x)\|_p dx \leq  \frac{\cC_{T,p}^{(1)}}{\sqrt{R}}\sum_{j=1}^{m}|b_j|\int_{-R}^R G_{t_j-r}(x-z) dx.
\]
Hence,
\begin{equation}
\label{XR-1}
\|D_{r,z}X_R D_{s,y} X_R\|_2 \leq \frac{\big(\cC_{T,4}^{(1)}\big)^2}{R}\sum_{j,k=1}^{m}|b_jb_k|\int_{[-R,R]^2} G_{t_j-r}(x_1-z) 
G_{t_k-r}(x_2-z)dx_1 dx_2.
\end{equation}

Similarly to \eqref{eq:D2FR-contr-bound},
\begin{align}
\nonumber
& \|[D^2 X_R \otimes_1 D^2 X_R]((r,z),(s,y))\|_2  \\
\nonumber
& \quad \leq  \frac{1}{R}\sum_{j,k=1}^{m} |b_j b_k|
\int_0^T \int_{\bR} \left\| \int_{-R}^{R} D_{(\theta,w),(r,z)}^2 u(t_j,x_3)dx_3 \right\|_4
\left\| \int_{-R}^{R} D_{(\theta,w),(r,z)}^2 u(t_k,x_4)dx_4 \right\|_4
d\theta dw\\
\nonumber
& \quad \leq \frac{\big(\cC_{T,4}^{(2)}\big)^2}{R} B\int_0^T \int_{\bR} \big(1+|r-\theta|^{-\frac{1}{2}}\big)\big(1+|s-\theta|^{-\frac{1}{2}}\big)G_{8T}(z-w)
G_{8T}(y-w)d\theta dw\\
\label{XR-2}
& \quad \leq \frac{\big(\cC_{T,4}^{(2)}\big)^2}{R} B \, (T+8T^{\frac{1}{2}})\,
 G_{16T}(z-y) \big(1+|r-s|^{-\frac{1}{2}}\big)
\end{align}
where $B=\sum_{j,k=1}^{m} |b_j b_k|$. Using \eqref{XR-1} and \eqref{XR-2}, we infer that
\begin{align*}
\cB_R & \leq \frac{\big(\cC_{T,4}^{(1)}\cC_{T,4}^{(2)}\big)^2}{R^2}B \, (T+8T^{\frac{1}{2}})
\sum_{j,k=1}^{m} |b_j b_k|\int_{[0,T]^2} \big(1+|r-s|^{-\frac{1}{2}}\big)\\
& \qquad \qquad \qquad \qquad  \left(\int_{[-R,R]^2} \int_{\bR^2}G_{16T}(z-y) G_{t_j-r}(x_1-z) 
G_{t_k-r}(x_2-z) dydz dx_1 dx_2 \right) drds.
\end{align*}
By the semigroup property, the inner integral is equal to
\[
\int_{[-R,R]^2} \int_{\bR^2}G_{16T+t_j+t_k-r-s}(x_1-x_2)dx_1 dx_2 \leq R.
\]
Hence, by \eqref{rs-int},
\begin{align*}
\cB_R & \leq \frac{\big(\cC_{T,4}^{(1)}\cC_{T,4}^{(2)}\big)^2}{R}B^2 \, (T+8T^{\frac{1}{2}}) 
\int_{[0,T]^2} \big(1+|r-s|^{-\frac{1}{2}}\big)drds \leq C_T B^2 R^{-1},
\end{align*}
where $C_T=\big(\cC_{T,4}^{(1)}\cC_{T,4}^{(2)}\big)^2 (T+8T^{\frac{1}{2}}) (T^2+4T^{\frac{3}{2}})$. Using \eqref{Vido1}, we conclude that \eqref{fdd} holds.

\medskip

{\bf Step 2.} (Existence of {H\"older} continuous modification and tightness)

We will use Kolmogorov-Chentsov criterion (Theorem 3.23 of \cite{kallenberg02}) to prove that $\{F_R(t)\}_{t \in [0,T]}$ has a {$\gamma$-H\"older continuous modification of order $\gamma \in (0,\frac{1}{2})$} (denoted also $F_R$), and the classical tightness criterion (Corollary 16.9 of \cite{kallenberg02}) to show that the family $\{R^{-1/2}F_R(\cdot)\}_{R>0}$ is tight in the space $C[0,T]$ of continuous functions on $[0,T]$.

\medskip

$u(t,x) -\bE u(t,x)= v(t,x) + Z(t,x)$,
where
\begin{align*}
v(t,x) &= \int_0^t \int_{\bR} G_{t-s}(x-y) \big(b(u(s,y)) -\bE b(u(s,y))\big)dyds,\\
Z(t,x) &= \int_0^t \int_{\bR} G_{t-s}(x-y)\sigma\big(u(s,y)\big)W(ds,dy).
\end{align*}

Using this notation,

\[
F_R(t) = \int_{-R}^R v(t,x)  dx + \int_{-R}^R Z(t,x)dx=:A_R(t)+B_R(t).
\]
We will apply Kolmogorov-Chentsov criterion and the tightness criterion to $A_R$ and $B_R$ separately. Then, we will use the fact that the sum of two $\gamma$-H\"older continuous processes is also $\gamma$-H\"older continuous, and the sum of two tight families in $C[0,T]$ is also tight.

\medskip

We examine $A_R$ first.
Note that {$v$} solves the PDE:
\[
\frac{\partial v}{\partial t}(t,x) = \frac{1}{2}\frac{\partial^2 v}{\partial x^2}(t,x) + \big(b(u(t,x) - \bE b(u(t,x)))\big).
\]
The spatial integral $A_R$ {of $v$} is differentiable and satisfies
\begin{align*}
\frac{d}{dt}A_R(t)& =\int_{-R}^R 	\frac{\partial v}{\partial t} (t,x)dx = \frac{1}{2}\int_{-R}^R \frac{\partial^2 {v}}{\partial x^2}(t,x)dx + \int_{-R}^R \big(b(u(t,x)) - \bE b(u(t,x))\big)dx\\
	&=\frac{1}{2}\left(\frac{\partial v}{\partial x}(t,R) - \frac{\partial v}{\partial x}(t,-R)\right) + \int_{-R}^R \big(b(u(t,x)) - \bE b(u(t,x))\big)dx\\
	&=:T_1(t) + T_2(t).
\end{align*}

{For any $0\leq s<t\leq T$,
$A_R(t)-A_R(s)=\int_s^t \frac{d}{dr}A_R(r)dr=\int_s^t \big(T_1(r)+T_2(r)\big)dr$,
and}
\begin{equation}
\label{est-A}
\|A_R(t)-A_R(s)\|_2 \leq \int_s^t \big( \|T_1(r)\|_2+\|T_2(r)\|_2 \big) dr.
\end{equation}
We estimate the $L^2(\Omega)$-norm of $T_1(t)$. By Minkowski's inequality,  for any $x \in \bR$,
\begin{align*}
 \left\|\frac{\partial v}{\partial x}(t,{x})\right\|_2 & \leq \int_0^t \int_{\mathbb{R}} \left|\frac{\partial G_{t-s}}{\partial x}(x-y) \right|\|b(u(s,y)) - \bE b(u(s,y))\|_2dyds\\
	&\leq 2L_bK_{t,2} \int_0^t \int_{\mathbb{R}} \left|\frac{x-y}{\sqrt{2\pi} (t-s)^{\frac{3}{2}}}\right|e^{-\frac{(x-y)^2}{2(t-s)}}dyds\\
	&\leq {2} L_b K_{t,2} \int_0^t (t-s)^{-\frac{1}{2}}ds{=4} L_b K_{t,2} t^{\frac{1}{2}}. 
	\end{align*}
{the upper bound for this term does not depend on $x$. In particular, applying this bound for $x=\pm R$, we obtain that:
\begin{equation}
\label{T1}
\|T_1(t)\|_2 \leq 4 L_b K_{t,2} t^{\frac{1}{2}}. 
\end{equation}
}

Next we estimate the $L^2(\Omega)$-norm of $T_2(t)$. {By the chain rule \eqref{chain},
$D_{r,z}T_2(t)=\int_{-R}^R b'(u(t,x))  D_{r,z} u(t,x)dx$.}
 Using Poincar\'e inequality \eqref{Poincare},
\begin{align*}
&\|T_2(t)\|_2^2 = {\rm Var}\big(T_2(t)\big) \leq \bE\|T_2(t)\|_{\cH}^2 =\bE \int_0^t \int_{\bR}|D_{r,z}T_2(t)|^2 dzdr \\
	&=\int_0^t \int_{\bR} \int_{-R}^R\int_{R}^R \bE \big[b'(u(t,x) b'(u(t,y) D_{r,z}u(t,x) D_{r,z}u(t,y) \big] dx dy dzdr \\
	&\leq L_b^2 \int_0^t \int_{\bR} \int_{-R}^R \int_{-R}^R \|D_{r,z}u(t,x)\|_2\|D_{r,z}u(t,y)\|_2dxdydzdr.
\end{align*}
We use \eqref{key} to bound this expression, {followed by the semigroup property:}
\begin{align}
\nonumber
\|T_2(t)\|_2^2 &\leq L_b^2 {\big(\cC_{t,2}^{(1)} \big)^2} \int_0^t \int_{\mathbb{R} }\int_{-R}^R \int_{-R}^R G_{t-r}(x-z)G_{t-r}(y-z)dxdydzdr\\
\nonumber
	& {= L_b^2 \big(\cC_{t,2}^{(1)} \big)^2 } \int_0^t \int_{-R}^R \int_{-R}^R G_{2(t-r)}(x-y)dxdy\\
\label{T2}
	&\leq L_b^2 \big(\cC_{t,2}^{(1)} \big)^2 \cdot 2Rt.
\end{align}
Combining estimates \eqref{est-A}, \eqref{T1} and \eqref{T2}, we see that for any $0<s<t\leq T$,
{
\begin{align} 
\nonumber
\|A_R(t)-A_R(s)\|_2 & \leq L_b \Big(4K_{T,2}+ \cC_{T,2}^{(1)}\sqrt{2R}\Big)\int_s^t \sqrt{r}dr \\
\label{eq:int-v-bound}
& \leq  2L_b \Big(4K_{T,2}+ \cC_{T,2}^{(1)})\Big)T^{1/2} R^{1/2} (t-s).
\end{align} 
In the previous line, we used the fact that $\sqrt{r} \leq T^{\frac{1}{2}}$ for all $r \leq T$.
Hence,
\[
\bE|A_R(t)-A_R(s))|^2 \leq 4L_b^2 \Big(4K_{T,2}+ \cC_{T,2}^{(1)}\Big)^2 T R (t-s)^2.
\]
By the two criteria mentioned above, we infer that $\{A_R(t)\}_{t \in [0,T]}$ has a $\gamma$-continuous modification for any $\gamma \in (0,\frac{1}{2})$, and the family $\{R^{-1/2}A_R(\cdot)\}_{R>0}$ is tight in $C[0,T]$. }

\medskip
 
{Next, we examine $B_R$.}
For any $0\leq r<t$ and $y\in \bR$, we denote
\[
\varphi_{t,R}(r,y)=\int_{-R}^{R}G_{t-r}(x-y)dx.
\]
We fix $s,t \in [0,T]$ with $s<t$. {We write $B_R(t)-B_R(s)= T_3 + T_4$,} where	
\begin{align*}
T_3 & =\int_0^s \int_{\bR} \big(\varphi_{t,R}(r,z)-\varphi_{s,R}(r,z)\big)\s\big(u(r,z)\big) W(dr,dz) \\
T_4 & =\int_s^t \int_{\bR} \varphi_{t,R}(r,z) \s\big(u(r,z)\big) W(dr,dz).
\end{align*}


{Recall that $K_{T,p,D_{\s}}$ is given by \eqref{s-u-bded}.}
We use BDG inequality \eqref{BDG}, followed by Minkowski inequality:
\begin{align*}
& \|T_3\|_p^2+\|T_4\|_p^2  \leq z_p^2 K_{T,2,D_{\s}}^2 
\left(\int_0^s\int_{\bR}|\varphi_{t,R}(r,z)-\varphi_{s,R}(r,z)|^2drdz+
\int_s^t\int_{\bR}\varphi_{t,R}^2(r,z)drdz\right)\\
&\quad = \frac{ z_p^2 K_{T,2,D_{\s}}^2  }{2\pi}\left(\int_0^s\int_{\bR}|\cF\varphi_{t,R}(r,\cdot)(\xi)-\varphi_{s,R}(r,\cdot)(\xi))|^2
d\xi dr+
\int_s^t\int_{\bR}|\cF\varphi_{t,R}(r,\cdot)(\xi)|^2 d\xi dr\right),
\end{align*}
using Plancherel theorem for the last line. Here $\cF \varphi(\xi)=\int_{\bR}e^{-i\xi x} \varphi(x) dx$ denotes the Fourier transform of $\varphi$. Since
\[
\cF\varphi_{t,R}(r,\cdot)(\xi)=\int_{-R}^R \cF G_{t-r}(x-\cdot)(\xi)dx=e^{-\frac{(t-r)|\xi|^2}{2}}\, \frac{\sin(R|\xi|)}{|\xi|},
\]
it follows that
\begin{align*}
 \|T_3\|_p^2+\|T_4\|_p^2  & \leq \frac{z_p^2 K_{T,p,D_{\s}}^2}{2\pi} \Bigg(
\int_0^s \int_{\bR} \Big( e^{-\frac{(t-r)|\xi|^2}{2}} -e^{-\frac{(s-r)|\xi|^2}{2}} \Big)^2  \frac{\sin^2(R|\xi|)}{|\xi|^2} d\xi dr+ \\
& \qquad \qquad \qquad  
\int_s^t \int_{\bR} e^{-(t-r)|\xi|^2} \, \frac{\sin^2(R|\xi|)}{|\xi|^2} d\xi dr\Bigg)\\
& \leq \frac{z_p^2 K_{T,p,D_{\s}}^2}{2\pi} C_T^2 R(t-s),
\end{align*}
where for the last line we used the estimates given on page 7182 of \cite{HNV20}. Hence,
\[
\|B_R(t)-B_R(s)\|_p \leq \|T_3\|_p+\|T_4\|_p \leq \frac{z_p K_{T,p,D_{\s}}}{\sqrt{2\pi}}  C_T R^{1/2}(t-s)^{1/2},
\]
or equivalently,
\[
\bE|B_R(t)-B_R(s)|^p \leq \left(\frac{z_p K_{T,p,D_{\s}}}{\sqrt{2\pi}}  C_T \right)^p R^{p/2}(t-s)^{p/2}.
\]

Choosing an arbitrary $p>2$ allows us to apply the two criteria mentioned above, to infer that $\{B_R(t)\}_{t \in [0,T]}$ has a $\gamma$-continuous modification for any $\gamma \in (0,\frac{1}{2})$, and the family $\{R^{-1/2}B_R(\cdot)\}_{R>0}$ is tight in $C[0,T]$.

\bigskip

{\em Acknowledgment.} (i) The first author is grateful for the invitation to visit Boston University in November 2023 when this project was initiated. (ii) The second author acknowledges support from the National Science Foundation under Grant No. DMS-1928930, while in residence at the Simons Laufer Mathematical Sciences Institute in Berkeley, California, during the fall semester of 2025.

\appendix

\section{Auxiliary Results}

{
In this section, we include the approximation result and the Gronwall-type lemma, which were used in the proofs of Theorems \ref{thm:1st-deriv-prelim} and \ref{key-D2-th}. }

We start with the approximation result. Recall that a Radon measure is finite on compact sets.

\begin{lemma}
\label{lem:weak-implies-pointwise}
Let $H=L^2(E,\cE,\mu)$, where $(E,\cE,\mu)$ is a finite-dimensional normed space, {$\cE$ is the Borel $\sigma$-field, and $\mu$ is a Radon measure on $(E,\cE)$.} Assume that the sequence $(X_n)_{n\geq 1} \subseteq L^2(\Omega;H)$ converges to $X$ in the weak topology of $L^2(\Omega;H)$. If there exist {$p> 1$} and a measurable function $\phi:E \to {[0,\infty)}$ such that
\begin{equation}
\label{ineq-Xn}
\sup_{n\geq 1}\|X_n(x)\|_p \leq \phi(x) \quad {\mbox{for $\mu$-almost all $x\in E$},}
\end{equation}
then
\[
\|X(x)\|_p \leq \phi(x) \text{ for $\mu$-almost all $x \in E$}.
\]
\end{lemma}

\begin{proof}
Define the sets
\begin{align*}
	A_{\e,R}&:= \{x \in E; \|x\|_E \leq R,\ \phi(x) + \e \leq \|X(x)\|_p \text{ and } \phi(x) \leq R\}\\
	B_N & := \left\{(\omega,x) \in \Omega \times E;  |X(\omega,x)|\leq N \|X(x)\|_p  \right\}.
\end{align*}
Note that $\mu(A_{\e,R})<\infty$,
since $A_{\e,R}$ is contained in the closed ball $\{x\in E;\|x\|_E \leq R\}$ (which is compact, since $E$ is finite-dimensional), and $\mu$ is Radon.

Define for $(\omega,x) \in \Omega \times E$,
\[Y(\omega,x) : = 1_{A_{\e,R}}(x) 1_{B_N}(\omega,x)X(\omega,x)|X(\omega,x)|^{p-2} \|X(x)\|_p^{1-p}.\]
This process is designed to have the property that for each fixed $x \in E$,
\begin{equation} \label{eq:Y-L-p-p-1}
	\|Y(x)\|_{\frac{p}{p-1}} \leq 1_{A_{\e,R}}(x).
\end{equation}
{
Indeed,
\begin{align*}
\bE|Y(x)|^{\frac{p}{p-1}}&=1_{A_{\e,R}}(x) \bE \big[ 1_{B_N}(x) |X(x)|^p\big] \ \|X(x)\|_{p}^{-p} \\
& \leq 1_{A_{\e,R}}(x) \bE  |X(x)|^p \ \|X(x)\|_{p}^{-p} =1_{A_{\e,R}}(x).
\end{align*}
}

The definition of $B_N$ also guarantees that the process is bounded pointwise:
\[
|Y(\omega,x)| \leq{ N^{p-1}} 1_{A_{\e,R}}(x).
\]
An immediate consequence of this pointwise bound is that $Y \in L^2(\Omega;H)$, since
\[
\bE \int_E |Y(x)|^2 \mu(dx) \leq {N^{2(p-1)}} \mu(A_{\e,R})<\infty.
\]


By H\"older's inequality and relations \eqref{ineq-Xn} and \eqref{eq:Y-L-p-p-1}, for each $n\geq 1$ we have:
\begin{align*}
& \bE \int_{E} X_n(x)Y(x)\mu(dx) \leq 
\int_{E} \big|\bE\big[X_n(x)Y(x)\big]\big| \mu(dx) \\
& \leq \int_{E} \|X_n(x)\|_p \|Y(x)\|_{\frac{p}{p-1}}\mu(dx) \leq { \int_{A_{\e,R}} \phi(x)\mu(dx)} \\
\end{align*}
Importantly, the defintion of $A_{\e,R}$ guarantees that
\[\int_{A_{\e,R}} \phi(x)\mu(dx) \leq R \mu(A_{\e,R})<\infty.\]
From the definition of $Y$,
\begin{align*}
\bE\int_{E} X(x)Y(x)\mu(dx) & = \int_{A_{\e,R}} \bE\big[1_{B_N}(x)|X(x)|^p\big] \cdot \| X(x)\|_p^{1-p}\mu(dx).
\end{align*}

By weak convergence,
$
\bE \int_{E} X_n(x)Y(x)\mu(dx)\to \bE \int_{E} X(x)Y(x)\mu(dx)$.
This leads to the statement
\[
\int_{A_{\e,R}} \bE\big[1_{B_N}(x)|X(x)|^p\big] \cdot \| X(x)\|_p^{1-p}\mu(dx) \leq \int_{A_{\e,R}}\phi(x)\mu(dx). 
\]

Letting $N \to \infty$, we obtain: $
\int_{A_{\e,R}} \|X(x)\|_p \mu(dx) \leq \int_{A_{\e,R}}\phi(x)\mu(dx)$.
By the definition of $A_{\e,R}$, 
\[
\int_{A_{\e,R}} (\phi(x) + \e)\mu(dx) \leq \int_{A_{\e,R}}\phi(x)\mu(dx), 
\]
which proves that $\mu\left(A_{\e,R}\right)=0$ because $\int_{A_{\e,R}}\phi(x)\mu(dx)<\infty$. Letting $\e \to 0$ and $R \to \infty$, we conclude that
$\|X(x)\|_p \leq \phi(x) $ for $\mu$-almost all $x \in E$.
\end{proof}


\begin{lemma}
\label{lemA2-HNV}
Fix $r>0$ and let $\rho(t,x) = \int_{\mathbb{R}} G_{t-r}(x-y)\mu(dy)$ be a solution to the heat equation for $t>r$, where $\mu$ is a non-negative  measure on $\bR$. {Assume that $I=(r,T]$ or $I=(r,\infty)$.}
Let $f_n:I \times \mathbb{R}\to \bR_+$ be a sequence of functions such that for any $t \in I$ and $x \in \bR$,
$f_0(t,x)=0$
and for any $n \geq 1$,
\begin{equation}
	\label{ineq-f}
	f_{n+1}^2(t,x) \leq C_t \left\{A^2\rho^2(t,x) + \left(\int_r^t \int_{\mathbb{R}} G_{t-s}(x-y)f_n(s,y)dyds \right)^2+\int_r^t \int_{\bR} G_{t-s}^2(x-y)f_n^2(s,y)dyds  \right\},
\end{equation}
where $A>0$ and $C_t>0$ is a constant which is non-decreasing in $t$, both of which are independent of $n$. Then for any $n \geq 1$, $t\in I$ and $x \in \bR$,
\begin{equation}
\label{fn-bound}
f_n(t,x) \leq \cC_t A \rho(t,x), 
\end{equation}
where
\[
\cC_t= 3^{\frac{1}{3}} C_t^{1/2}  \exp\left\{\frac{3C_t^3}{2} \left[(t-r)^6 + 
\frac{\Gamma(1/4)^4}{8\pi^{5/2}}
(t-r)^{\frac{3}{2}}\right] \right\}.
\]
\end{lemma}

\begin{proof}
For all $n\geq 0$ and $t\in I$, define
\[
\widetilde{f}_n(t) := \sup_{x \in \mathbb{R}} \frac{f_n(t,x)}{\rho(t,x)} \in [0,\infty].
\]

\medskip

{\em Step 1.} In this step, we assume that $\widetilde{f}_n(t)<\infty$ for all $t\in I$ (for some fixed $n\geq 0$), and we derive a recurrence relation for $\widetilde{f}_{n+1}$. 
By the semigroup property, we have:
\begin{align} 
	\nonumber
	\int_{\bR^d} G_{t-s}(x-y)f_n(s,y)dy  & \leq \widetilde{f}_n(s) \int_{\bR^d}G_{t-s}(x-y)\rho(s,y)dy \\
	\nonumber
	& =\widetilde{f}_n(s) \int_{\bR} \int_{\bR} G_{t-s}(x-y)G_{s-r}(y-z) \mu(dz)dy\\
	\label{eq:convolution-bound}
	&=\widetilde{f}_n(s) \int_{\bR}G_{t-r}(x-z)\mu(dz)=\widetilde{f}_n(s) \rho(t,x).
\end{align}

Next we bound the $ \int_{\mathbb{R}} G^2_{t-s}(x-y)f^2_n(s,y)dy$ integral using a duality argument:
\begin{align*}
	\int_{\bR^d} G_{t-s}^2(x-y)f^2_n(s,y)dy & 
\leq \widetilde{f}_n^2(s) \int_{\bR^d}G_{t-s}^2(x-y)\rho^2(s,y)dy\\
	&=\widetilde{f}_n^2(s)  \sup_{\|h\|_{L^2(\bR)}=1} \left(\int_\bR \int_\bR G_{t-s}(x-y)G_{s-r}(y-z)\mu(dz)h(y) dy\right)^2.
\end{align*}
Use the Cauchy-Schwarz inequality on the $dy$ integral to bound 
\begin{align*}
	&\sup_{\|h\|_{L^2(\bR)}=1} \left(\int_\bR \int_\bR G_{t-s}(x-y)G_{s-r}(y-z)\mu(dz)h(y)dy\right)^2\\
	&\leq \left(\int_\bR \left(\int_\bR G^2_{t-s}(x-y)G^2_{s-r}(y-z)dy\right)^{\frac{1}{2}}\mu(dz)\right)^2.
\end{align*}
This previous step required $\mu$ to be a nonnegative measure. Then {we use \eqref{G2-G2}} to evaluate the $dy$ integral.
Therefore, we can bound
\begin{equation} \label{eq:convolution-2-bound}
	\int_{\bR^d} G_{t-s}^2(x-y)f^2_n(s,y)dy \leq \widetilde{f}^2_n(s)\left(\frac{t-r}{4\pi(t-s)(s-r)}\right)^{\frac{1}{2}}\rho^2(t,x).
\end{equation}

Then applying \eqref{eq:convolution-bound} and \eqref{eq:convolution-2-bound} to \eqref{ineq-f} leads to the estimate
\begin{align*}
	f^2_{n+1}(t,x) & \leq C_t \rho^2(t,x) \left\{ A^2+ \left(\int_r^t \widetilde{f}_n(s)ds \right)^2+
	\frac{1}{\sqrt{4\pi}} \int_r^t \left(\frac{t-r}{(t-s)(s-r)} \right)^{1/2} \widetilde{f}_n^2(s)ds 
	\right\}.
\end{align*}

Divide by $\rho^2(t,x)$ and take the supremum for $x \in \bR$. We get:
\begin{equation}
	\label{eq:tilde-f-squared}
	\tilde{f}_{n+1}^2(t)  \leq C_t \left\{ A^2+ \left(\int_r^t \widetilde{f}_n(s)ds \right)^2+
	\frac{1}{\sqrt{4\pi}} \int_r^t \left(\frac{t-r}{(t-s)(s-r)} \right)^{1/2} \widetilde{f}_n^2(s)ds 
	\right\}.
\end{equation}
We will bound the 6th power of $\tilde{f}_{n+1}(t)$, but any power bigger than 4 will work. By H\"older's inequality,
\begin{equation}
	\int_r^t \tilde{f}_n(s)ds\leq (t-r)^{\frac{5}{6}} \left(\int_r^t \tilde{f}_n^6(s)ds\right)^{\frac{1}{6}}
\end{equation}
and
\begin{align}
	&\int_r^t  \left(\frac{t-r}{4\pi(t-s)(s-r)}\right)^{\frac{1}{2}}\tilde{f}_n^2(s)ds \nonumber\\
	&\leq \left(\int_{r}^t\left(\frac{t-r}{4\pi(t-s)(s-r)}\right)^{\frac{3}{4}}ds\right)^{\frac{2}{3}}\left(\int_r^t\tilde{f}_n^6(s)ds\right)^{\frac{1}{3}} \nonumber\\
	&\leq (4\pi)^{-\frac{1}{2}}\left(B\left(\frac{1}{4},\frac{1}{4}\right) (t-r)^{\frac{1}{4}} \right)^{\frac{2}{3}} \left(\int_r^t \tilde{f}_n^6(s)ds \right)^{\frac{1}{3}},
\end{align}
where $B\left(\frac{1}{4},\frac{1}{4}\right)$ is a Beta function.

Therefore, raising \eqref{eq:tilde-f-squared} to the third power,
\begin{equation}
\label{rec}
	\tilde{f}^6_{n+1}(t) \leq 3^2{C_t^3} \left( A^6  + \left((t-r)^5  + {(4\pi)^{-\frac{3}{2}}}B\left(\frac{1}{4},\frac{1}{4}\right)^{2}
(t-r)^{\frac{1}{2}} \right) \int_r^t \tilde{f}_n^6(s)ds\right).
\end{equation}

%
%

{\em Step 2.} In this step, we prove by induction on $n\geq 0$ that 
\begin{equation}
\label{induction}
\widetilde{f}_n(t)  <\infty \quad \mbox{for all $t\in I$}.
\end{equation}

The statement holds trivially for $n=0$ since $\tilde{f}_0(t)\equiv 0$. 
For the induction step, we assume that \eqref{induction} holds
for all $k=0,\ldots,n$, and we show that it holds for $n+1$. By Step 1, the recurrence relation \eqref{rec} holds for all $k=0,\ldots,n$.

We fix $t\geq r$ and set $$\alpha = 3^2C_t^3A^6 \quad \mbox{and} \quad \beta = 3^2C_t^3 \left[(t-r)^5 +(4\pi)^{-\frac{3}{2}} B \left(\frac{1}{4},\frac{1}{4}\right)^2(t-r)^{\frac{1}{2}}\right].$$
Then for all $k=0,\ldots,n$ and $s\in [r,t]$,
\[
\widetilde{f}_{k+1}^6(s) \leq \alpha+\beta \int_r^t \widetilde{f}_{k}^6(s)ds. 
\]

By Lemma \ref{gron-classic} below, we infer that for any $k=1,\ldots,n+1$ and $s \in [r,t]$,
$$\tilde{f}_k^6(s) \leq \alpha e^{\beta(s-r)}.$$
In particular, $\tilde{f}_{n+1}(t) <\infty$.

\medskip

{\em Step 3.} Having proved that \eqref{induction} holds for all $n\geq 0$, we can now invoke 
the classical Gronwall lemma (e.g. \cite[Lemma 10.2.4]{kuo06}) on the fixed interval $[r,t]$, for arbitrary $t\in I$. We infer that for all $n\geq 1$ and $s \in [r,t]$,
\[
\tilde{f}_n^6(s) 
\leq 3^2 C_t^3 A^6 \exp\left\{ 3^2C_t^3 \left[(t-r)^5 + (4\pi)^{-\frac{3}{2}}B \left(\frac{1}{4},\frac{1}{4}\right)^2(t-r)^{\frac{1}{2}}\right](s-r)\right\}.
\]

In particular, for $s=t$ 
\begin{equation*}
\tilde{f}_{n}^6(t) 
\leq 3^2 C_t^3 A^6 \exp\left\{ 3^2C_t^3 \left[(t-r)^6 + (4\pi)^{-\frac{3}{2}}B \left(\frac{1}{4},\frac{1}{4}\right)^2(t-r)^{\frac{3}{2}}\right]\right\}.
\end{equation*}

Taking power $1/6$, we conclude that for any $n\geq 1$ and $t\in I$,
\begin{equation*}
	\tilde{f}_n(t) \leq 3^{\frac{1}{3}} C_t^{1/2}  A \exp\left\{\frac{3C_t^3}{2} \left[(t-r)^6 + (4\pi)^{-\frac{3}{2}}B \left(\frac{1}{4},\frac{1}{4}\right)^2(t-r)^{\frac{3}{2}}\right] \right\}.
\end{equation*}

(Note: For $n=1$, $\widetilde{f}_1^2(t) \leq C_t A^2$ gives a sharper bound.)
The conclusion follows using the fact that $B \left(\frac{1}{4},\frac{1}{4}\right)=\frac{\Gamma(1/4)^2}{\sqrt{\pi}}$.

\end{proof}

\begin{lemma}
\label{gron-classic}
Let $(f_n)_{n=0,\ldots,N+1}$ be a sequence of non-negative functions on $[a,b]$ such that for any $n= 0, \ldots, N$ and $t\in [a,b]$
\[
f_{n+1}(t) \leq \alpha(t)+\int_a^t \beta(s) f_n(s)ds
\]
for some non-negative functions $\alpha$ and $\beta$ on $[a,b]$. We assume that $f_n(t)<\infty$ for all $n=0,\ldots,N$, but $f_{N+1}(t)$ may be $\infty$.
Then for any $n=1,\ldots,N+1$ and $t \in [a,b]$
\[
f_n(t)\leq \alpha(t)+\int_0^t \big(\alpha(s) +f_0(s)\big)\beta(s) \exp\left(\int_s^t\beta(r)dr\right)ds.
\]
(Note: For $n=1$, we have the sharper bound: $f_{1}(t) \leq \alpha(t)+\int_a^t \beta(s) f_0(s)ds$.)

In particular, if $\alpha$ is non-decreasing and $f_0=0$, then for any $n=1,\ldots, N+1$ and $t \in [a,b]$,
\[
f_n(t) \leq \alpha(t) \exp \left(\int_a^t \beta(s)ds \right),
\]
and hence $f_{N+1}(t)<\infty$ for all $t \in [a,b]$.
\end{lemma}

\begin{proof} Using the recurrence formula for $f_1$, we see that
\begin{align*}
f_2(t)& \leq \alpha(t)+\int_a^t \beta(s_1)f_1(s_1)ds_1 \\
&\leq \alpha(t)+\int_a^t \beta(s_1) \alpha(s_1)ds_1+\int_{a<s_2<s_1<t} \beta(s_1) \beta(s_2)f_0(s_2)ds_1 ds_2. 
\end{align*}
If we do one more iteration, we get:
\begin{align*}
f_3(t)& \leq \alpha(t)+\int_a^t \beta(s_1)f_2(s_1)ds_1 \\
&\leq \alpha(t)+\int_a^t \beta(s_1) \alpha(s_1)ds_1+\int_{a<s_2<s_1<t} \beta(s_1) \beta(s_2) \alpha(s_2)ds_1 ds_2+\\
& \quad \quad \quad \int_{a<s_3<s_2<s_1<t} \beta(s_1) \beta(s_2) \beta(s_3) f_0(s_3)ds_1 ds_2 ds_3.
\end{align*}

By induction, for any $n=1,\ldots, N+1$, we have:
\begin{align}
\label{ineq-fn-IR}
f_n(t)\leq \alpha(t)+\sum_{k=1}^{n-1}\cI_k(t)+\cR_n(t)
\end{align}

where
\begin{align*}
\cI_k(t)& =\int_{a<s_k<\ldots<s_1<t} \beta(s_1) \ldots \beta(s_k)\alpha(s_k)ds_1 \ldots ds_k\\
\cR_n(t)&=\int_{a<s_n<\ldots<s_1<t} \beta(s_1) \ldots \beta(s_n)f_0(s_n)ds_1 \ldots ds_n.
\end{align*}

Note that
\begin{align*}
\cI_k(t)&=\int_a^t  \alpha(s_k) \beta(s_k)\left(\int_{s_k<s_{k-1}<\ldots<s_1<t}\beta(s_1) \ldots \beta(s_{k-1})ds_1 \ldots ds_{k-1}\right)ds_k\\
&=\int_a^t  \alpha(s) \beta(s) \frac{1}{(k-1)!} \left(\int_{s}^t \beta(r)dr\right)^{k-1}ds
\end{align*}
and hence
\begin{align}
\nonumber
\sum_{k=1}^{n-1} \cI_k(t) & \leq \int_a^t  \alpha(s) \beta(s) \sum_{k\geq 1}\frac{1}{(k-1)!} \left(\int_{s}^t \beta(r)dr\right)^{k-1}ds \\
\label{ineq-sum-Ik}
& \leq \int_a^t  \alpha(s) \beta(s)  \exp\left(\int_{s}^t \beta(r)dr\right)ds.
\end{align}

Similarly,
\begin{align}
\nonumber
\cR_n(t)&=\int_a^t  f_0(s_n) \beta(s_n)\left(\int_{s_n<s_{n-1}<\ldots<s_1<t}\beta(s_1) \ldots \beta(s_{n-1})ds_1 \ldots ds_{n-1}\right)ds_n\\
\nonumber
&=\int_a^t  f_0(s) \beta(s) \frac{1}{(n-1)!} \left(\int_{s}^t \beta(r)dr\right)^{n-1}ds\\
\label{ineq-Rn}
& \leq \int_a^t  f_0(s) \beta(s) \exp\left(\int_{s}^t \beta(r)dr\right)ds
\end{align}

Inserting estimates \eqref{ineq-sum-Ik} and \eqref{ineq-Rn} into \eqref{ineq-fn-IR}
proves the first statement.

Assume now that $\alpha$ is non-decreasing and $f_0=0$ for all $t$. Then
\[
f_n(t)\leq \alpha(t)+\alpha(t) \int_a^t \beta(s)\exp\left(\int_{s}^t \beta(r)dr\right)ds=\alpha(t)\exp\left(\int_a^t \beta(r)dr\right).
\]
\end{proof}

\end{document}